%% file: main.tex
\documentclass[12 pt]{amsart}

\title[MCS formulation with weakly imposed symmetry]{A mass conserving mixed stress formulation for Stokes flow with weakly imposed stress symmetry}

\thanks{% Submitted. % to the editors DATE.
Philip L. Lederer has been funded by the Austrian Science Fund (FWF) through the research program ``Taming complexity in partial differential systems'' (F65) - project ``Automated discretization in multiphysics'' (P10).}

\author[J.~Gopalakrishnan]{Jay Gopalakrishnan}
\address{Portland State University, PO Box 751, Portland OR 97207,USA }
\email{gjay@pdx.edu}

\author[P.~L.~Lederer]{Philip L. Lederer}
\address{Institute for Analysis and Scientific Computing, TU Wien,
Wiedner Hauptstra{\ss}e 8-10, 1040 Wien, Austria}
\email{philip.lederer@tuwien.ac.at}

\author[J.~Sch\"oberl]{Joachim Sch{\"oberl}}
\address{Institute for Analysis and Scientific Computing, TU Wien,
Wiedner Hauptstra{\ss}e 8-10, 1040 Wien, Austria}
\email{joachim.schoeberl@tuwien.ac.at}

\usepackage{amsmath,amsfonts,amssymb}
\usepackage{bm}
\usepackage{todonotes}
\usepackage{pgfplotstable, pgfplots}
\usepackage{booktabs}
\usepackage{subcaption}

\usepackage{tikz}
\usetikzlibrary{shapes,arrows,snakes,calendar,matrix,backgrounds,folding,calc,positioning,spy,decorations.pathreplacing}

\usepackage{siunitx}
\newcommand{\numeoc}[1]{\num[round-precision=1,round-mode=places, scientific-notation=false]{#1}}
\usepackage{url}

\theoremstyle{plain}
\newtheorem{theorem}{Theorem}
\newtheorem{lemma}[theorem]{Lemma}
\newtheorem{corollary}[theorem]{Corollary}%[section]

\theoremstyle{remark}
\newtheorem{remark}[theorem]{Remark}

\input{macros}

\begin{document}

% \maketitle

% REQUIRED
\begin{abstract}
  % 
  % Removing all citations from abstract
  % 
  We introduce a new discretization of a mixed formulation of the
  incompressible Stokes equations that includes symmetric viscous stresses.
  The method is built upon a mass conserving mixed formulation that we
  recently studied. The improvement in this work is a new method
  that directly approximates the viscous fluid stress
  $\sigma$, enforcing its symmetry weakly.
  The finite element
  space in which the stress is approximated consists of 
  matrix-valued
  functions having continuous ``normal-tangential'' components across
  element interfaces.
  Stability is achieved by adding  certain matrix bubbles that were introduced
  earlier in the literature on finite elements for linear elasticity.  
  Like the earlier work, the new method here  approximates 
  the fluid velocity $u$ using $H(\divergence)$-conforming finite elements,
  thus providing exact mass conservation.
  Our error analysis shows optimal convergence rates
  for the pressure and the stress variables. An additional post processing yields
  an optimally convergent velocity satisfying exact mass conservation.
  The method is also pressure robust.
\end{abstract}

% REQUIRED
% \begin{keywords}
%   mixed finite element methods; incompressible flows; Stokes equations; weak symmetry
% \end{keywords}

% % REQUIRED
% \begin{AMS}
%   ???
% \end{AMS}

\maketitle

\section{Introduction}   \label{sec:introduction}

In this work we introduce a new method for the discretization of
steady incompressible Stokes system that includes symmetric viscous stresses. Let $\Omega \subset \rr^d$ be a bounded domain with $d =2$ or $3$ having a Lipschitz boundary $\Gamma := \partial \Omega$. Let $u$ and $p$ be the velocity and the pressure, respectively. Given an external body force $f:\om \to \RRR^d$ % \in L^2(\Omega)$ [must be vector???]
and kinematic viscosity $\tilde{\nu}:\om \to \RRR$, the
velocity-pressure formulation of the Stokes system is given by
\begin{align} \label{eq::stokes}
  \left\{
  \begin{aligned}
  -\divergence(2\tilde{\nu} \eps(\Velvar)) + \nabla \Presvar & = \Forcevar && \textrm{in } \Omega,  \\
  \divergence (\Velvar) &=0 && \textrm{in } \Omega, \\
  \Velvar &= 0 &&\textrm{on } \Gamma,
\end{aligned}\right.
\end{align}
where $\eps(\Velvar) =  (\nabla u + (\nabla u)^ \trans)/2$. By introducing a new variable $\sigma = \nu \eps (u)$ where $\nu := 2 \tilde{\nu}$, equation \eqref{eq::stokes} can be reformulated to 
\begin{subequations}
  \label{eq::stokesintro_two}
  \begin{alignat}{2}
  \label{eq::stokesintro_two-a}
  \frac{1}{\nu}{\Dev\Stressvar} - \eps( \Velvar)& = 0 \quad && \textrm{in } \Omega,  \\
  \label{eq::stokesintro_two-b}
 \divergence(  \Stressvar) - \nabla \Presvar & = -\Forcevar \quad && \textrm{in } \Omega,  \\
 \label{eq::stokesintro_two-c} 
  \divergence (\Velvar) &=0 \quad&& \textrm{in } \Omega, \\
  \label{eq::stokesintro_two-d}
  \Velvar &= 0 \quad && \textrm{on } \Gamma.
 \end{alignat}
\end{subequations}
We shall call formulation~\eqref{eq::stokesintro_two} the {\em {\bf{m}}ass
{\bf{c}}onserving mixed formulation with symmetric {\bf{s}}tresses}, or simply the MCS formulation.
Although formulations~\eqref{eq::stokes}
and~\eqref{eq::stokesintro_two} are formally equivalent, the MCS
formulation~\eqref{eq::stokesintro_two} demands less regularity of the
velocity field $u$.  Many authors have studied this formulation
previously \cite{MR1934446,MR1464150,MR1231323,Farhloulcanadian},
including us~\cite{gopledschoe}. In~\cite{gopledschoe}, following the
others, we introduced a new variable $\sigma = \nu\nabla u$, which is in
general nonsymmetric, and considered an analogous formulation (which
was also called an MCS formulation).
The main novelty in~\cite{gopledschoe} was that $\sigma = \nu\nabla u$ was set in a new function space $H(\curl \divergence, \Omega)$ of matrix-valued functions whose divergence can continuously act on elements of $H_0(\divergence, \Omega)$. Accordingly, the appropriate velocity space there was $H_0(\divergence, \Omega),$ not $H_0^1(\Omega, \rr^2)$ as in the classical velocity-pressure formulation.

In contrast to~\cite{gopledschoe}, in this work we set $\sigma = \nu
\eps(u)$, not $\nu\nabla u$. Our goal is to apply what we learnt
in~\cite{gopledschoe} to produce a new method that provides a direct
approximation to the {\em symmetric} matrix function $\sigma = \nu
\eps(u)$. Being the viscous stress, this $\sigma$ is of more direct
practical importance (than $\nu\nabla u$).  We shall seek $\sigma$ in
the same function space $H(\curl \divergence, \Omega)$ that we
considered in~\cite{gopledschoe}.
We have shown in~\cite{gopledschoe}
that matrix-valued
finite element functions with ``normal-tangential'' continuity across
element interfaces are natural for approximationg solutions in 
$H(\curl \divergence, \Omega).$ We shall continue to use such finite
elements here.  It is interesting to note that 
in the HDG (hybrid discontinuous Galerkin) literature~\cite{MR3194122,
  fujinqiu} the potential importance of such normal-tangential
continuity was noted and arrived at through a completely different
approach.

The main point of departure in this work, stemming from that fact
that $H(\curl \divergence, \Omega)$
contains non-symmetric matrix-valued functions, is that we  impose the
symmetry of stress approximations weakly using Lagrange multipliers.
This technique of imposing symmetry weakly is widely used in finite
elements for linear elasticity~\cite{MR840802, MR2336264,
  MR2449101,MR1464150}. In particular, our analysis is inspired by the
early work of Stenberg \cite{Stenberg1988}, who
enriched the stress space by curls of local element bubbles. (In fact,
this idea was even used in a Stokes mixed
method~\cite{MR1934446}, but their resulting method is not pressure robust.)
These enrichment curls lie in the kernel of the divergence operator and are only ``seen'' by the weak-symmetry constraint allowing them to be used to prove discrete inf-sup stability.
While in two dimensions -- assuming a triangulation into simplices -- this technique only increases the local polynomial order by $1$, this is not the case in three dimensions. Years later~\cite{MR2629995,GopalGuzma12}, it was realized that it is possible to retain the good convergence properties of Stenberg's construction and yet reduce the enrichment space. Introducing a ``matrix bubble,'' these works added just enough extra curls needed to prove stability.

% in \cite{Stenberg1988} but the number of additional basis functions is decreased
% Further, although realized years later, the number of additional basis functions that were added was sub optimal resulting in a computational overhead. In \cite{MR2629995} the authors tackled the question if it is possible to retain the good convergence properties of Stenbergs method in \cite{Stenberg1988} but the number of additional basis functions is decreased and the polynomial order (in three dimensions) is only increased by one. Introducing the so called {\it symmetric bubble matrix} the authors of \cite{MR2629995} were able to define certain curl-bubbles and achieved the desired properties.

We shall see in later sections that the matrix bubble can also be used to enrich our discrete fluid stress space. This might seem astonishing at first.  Indeed, an enrichment space for fluid stresses must map well when using a specific map that is natural to ensure normal-tangential continuity of the discrete stress space.  Moreover, the enrichment functions must lie in the kernel of a realization of the distributional row-wise divergence used in MCS formulations (displayed in~\eqref{eq:10} below). It turns out that these properties are all fulfilled by an enrichment using a double curl involving matrix bubbles. Hence we are able to prove the discrete inf-sup condition.  Stability then follows in the same type of norms used in \cite{Stenberg1988} and is a key result of this work.

Some comments on the choice of the discrete velocity space and its implications are also in order here. As mentioned above, the velocity space within the MCS formulation is $V = H_0(\divergence, \Omega)$. One of the main features of the
first MCS method  \cite{gopledschoe}, as well the new version with weakly imposed symmetry of this paper, is that we can choose a discrete velocity space $V_h \subset V$ using $H(\divergence)$-conforming finite elements. Therefore, our method is tailored to approximate the incompressibility constraint exactly, leading to pointwise and exactly divergence-free discrete velocity fields. The use of  such $H(\divergence)$-conforming velocities in Stokes flow is by no means new:
for the standard velocity-pressure formulation, once can find it in~\cite{cockburn2005locally, cockburn2007note},
and for the Brinkman Problem  in~\cite{stenberg2011}. Therein, and also in the
more recent works of~\cite{LS_CMAME_2016,LedererSchoeberl2017}, the $H^1$-conformity is treated in a weak sense and a (hybrid) discontinuous Galerkin method is constructed. When employing $H(\divergence)$-conforming finite elements, one has the luxury of choice. In~\cite{gopledschoe}, we used the 
$\BDM^{k+1}$ space \cite{brezzi1985two} and added several local stress bubbles in order to guarantee stability. In contrast, in this paper, we have chosen to take the smaller Raviart-Thomas space \cite{RaviaThoma77} of order $k$, denoted by $\RT^k$. A similar choice was made also in the work of~\cite{fujinqiu}, where   they presented a hybrid method for solving the Brinkman problem based off the work of~\cite{MR3194122}. Our current  choice of the smaller space $\RT^k$ leads to a less accurate velocity approximation (compared to $\BDM^{k+1}$), so in order to retain the optimal convergence order of the velocity (measured in a discrete $H^1$-norm), we introduce a local element-wise post processing. Using the reconstruction operator of \cite{ledlehrschoe2017relaxedpartI,ledlehrschoe2018relaxedpartII} this post processing can be done retaining the exact divergence-free property.

The remainder of this paper is organized as follows. In Section 2, we
define notation for common spaces used throughout this work and
introduce an undiscretized formulation. Section 3 presents the MCS
method for Stokes flow including symmetric viscous stresses.  In
Section 4, we present the new discrete method including the
introduction of the matrix bubble. Section 5 proves a discrete inf-sup
condition and develops a complete a priori error analysis of the
discrete MCS system.  In Section 6, we introduce a postprocessing for
the discrete velocity. The concluding section (Section 7) reports
various numerical experiments we performed to illustrate the theory.

\section{Preliminaries}  \label{sec:preliminaries}

In this section, we introduce notation and present a
weak formulation for Stokes flow that includes symmetric viscous stresses.

Let $\DD(\om)$ or $\DD(\om, \RRR)$
denote the set of infinitely differentiable compactly supported real-valued functions on $\om$ and let $\DD^*(\om)$ denote the space of distributions. To differentiate between scalar, vector and matrix-valued functions on $\om$, we include the co-domain in the notation, e.g., $\DD(\om, \RRR^d) = \{ u: \om \to \RRR^d| \; u_i \in \DD(\om)\}$. Let $\mm$ denote the vector space of real $d \times d $ matrices. This notation scheme is similarly 
extended to other function spaces as needed. Thus, $L^2(\om) = L^2(\om, \RRR)$ denotes the space of square integrable $\rr$-valued functions on $\om$, while  analogous vector and matrix-valued function spaces are defined by
$
  L^2(\om, \rr^d) := \left\{ u : \om \to \rr^d \big| \;u_i \in L^2(\om)\right\}$ and 
$  L^2(\om, {\mm}) := \left\{ \sigma : \om \to {\mm} \big|\; \sigma_{ij} \in L^2(\om)\right\}$, respectively. Let $\kk$ denote the vector space of $d \times d$ skew symmetric matrices, i.e., $\kk = \skw(\mm)$, and  let 
$L^2(\om,\kk) := \left\{ \sigma : \om \to \kk \big| \;\sigma_{ij} \in L^2(\om)\right\}$.

Recall that the dimension $d$ in this work is either 2 or 3. Accordingly, depending on the context, certain differential operators have different meanings. The ``curl'' operator, depending on the context,  denotes one of  the differential operators below.
\begin{align*}
  \curl(\phi)
  & = (-\partial_2 \phi, \partial_1 \phi)^\trans, 
  && \text{ for } \phi \in \DD^*(\om, \RRR), d=2,
  % \\
  % \curl( \phi)
  % & = -\partial_2 \phi_1+ \partial_1 \phi_2,
  % && \text{ for } \phi \in \DD^*(\om, \RRR^2), d=2,
  \\
  \curl (\phi)
  & 
    = 
    (\d_2\phi_3 - \d_3 \phi_2, \d_3\phi_1 - \d_1\phi_3, 
    \d_1\phi_2 - \d_2\phi_1)^\trans
  && \text{ for } \phi \in \DD^*(\om, \RRR^3), d=3,     
\end{align*}
where $(\cdot)^\trans$ denotes the transpose and $\d_i$ abbreviates
$\d/\d x_i.$ For matrix-valued functions in both  $d=2$ and $3$ cases, i.e., 
$\phi \in \DD^*(\om, \mm),$
by $\curl(\phi)$ we mean the matrix obtained by taking $\curl$ row wise. Unfortunately, this still does not exhaust all the curl cases. In the $d=2$ case, 
there are two possible definitions of $\curl(\phi)$
for  $\phi \in \DD^*(\om, \RRR^2)$,
\begin{align}
  \label{eq:curl-vec2scal-2d}
  \curl( \phi)
  & = -\partial_2 \phi_1+ \partial_1 \phi_2,
  \qquad  \text{ or }
  \\
  \label{eq:curl-vec2mat-2d}
  \curl (\phi)
  & 
    =
    \begin{pmatrix}
      \d_2 \phi_1  & -\d_1 \phi_1 \\
      \d_2 \phi_2  & -\d_1 \phi_2
    \end{pmatrix},  
\end{align}
and we shall have occasion to use both. The latter will not be used
until~\eqref{eq:deltaSigma} below, so until then, the reader may
continue assuming we mean~\eqref{eq:curl-vec2scal-2d} whenever we
consider curl of vector functions in $\rr^2$. The
operator $\grad$ is to be understood from context as an operator that
results in either a vector whose components are
$[\grad \phi]_i = \d_i \phi$ for $\phi \in \DD^*(\om, \RRR),$ or a
matrix whose entries are $[\grad \phi]_{ij} = \d_j \phi_i$ for
$\phi \in \DD^*(\om, \RRR^d),$ or a third-order tensor whose entries
are
$[\grad \phi]_{ijk} = \d_k \phi_{ij}$ for
$\phi \in \DD^*(\om, \kk).$
Finally, in a similar manner, we
understand $\div(\phi)$ as either $\sum_{i=1}^d \d_i \phi_i$ for
vector-valued $\phi \in \DD^*(\om, \RRR^d),$ or the row-wise
divergence $\sum_{j=1}^d \d_j \phi_{ij}$ for matrix-valued
$\phi \in \DD(\om,\mm)^*$.

Let $\dt= d(d-1)/2$ (so that $\dt=1$ and $3$ for $d=2$ and 3, respectively).
In addition to the standard Sobolev space $H^m(\om)$ for any $m\in \RRR,$
we shall use
the well-known space
$  \vecb{H}(\divergence,\Omega)
= \{ \Velvar \in L^2(\Omega, \rr^d): \divergence(\Velvar) \in
    L^2(\Omega) \}.$
% \begin{align*}
%   \vecb{H}(\divergence,\Omega)
%   &= \{ \Velvar \in L^2(\Omega, \rr^d): \divergence(\Velvar) \in
%     L^2(\Omega) \},
  % \\
  % \vecb{H}(\curl,\Omega) 
  % &= \{ \Velvar \in L^2(\Omega, \rr^d):
  %   \curl(\Velvar) \in L^2(\Omega,
  %   \rr^{\dt}) \}.
%\\
%  H^{-1}(\curl, \Omega)  
%  &= \{ \vecb{\phi} \in H^{-1}(\Omega, \rr^d): \curl(\vecb{\phi}) \in
%    H^{-1}(\Omega, \rr^{\dt})\}, 
%\\
%  \Hcurldiv{\Omega} 
%  & = \{ \Stressvar \in L^2(\Omega, {\mm}):
%    \curl(\divergence (\Stressvar)) \in H^{-1}(\Omega, \rr^{\dt}) \}.
%\end{align*}
By its trace theorem,
% of these  spaces we further define the subspaces with vanishing traces with a subscript $0$. E.~g., 
% the subspace of functions in $\vecb{H}(\divergence,\Omega)$ with vanishing normal trace is given by
$H_0(\div,\om) = \{ u \in H(\div, \om): u \cdot n|_\Gamma =0\}$ is a well-defined closed subspace, where $n$ denotes the outward unit normal on $\Gamma$.
Its dual space $[H_0(\divergence,\Omega)]^*$, as proved
in \cite[Theorem 2.1]{gopledschoe}, satisfies  
\begin{align}
  \label{eq:H0div-Hcurl}
[H_0(\divergence,\Omega)]^* = H^{-1}(\curl, \Omega) = \{ \vecb{\phi} \in H^{-1}(\Omega, \rr^d): \curl(\vecb{\phi}) \in H^{-1}(\Omega, \rr^{\dt})\}.
\end{align}
In this work, the following space is important:
\begin{align*}
\Hcurldiv{\Omega} 
  & := \{ \Stressvar \in L^2(\Omega, {\mm}):
    \divergence (\Stressvar) \in [H_0(\divergence,\Omega)]^* \},
\end{align*}
where the name results from~\eqref{eq:H0div-Hcurl}: indeed
 a function $\sigma \in \Hcurldiv{\Omega}$ fulfills $\curl\divergence(\sigma) \in H^{-1}(\Omega, \rr^{\dt})$.

Next, let us derive a variational formulation of the system \eqref{eq::stokesintro_two}, which is based on the mixed stress formulation (MCS) introduced in chapter 3 in the work \cite{gopledschoe}. The method is based on a weaker regularity assumption of the velocity as compared to the standard velocity-pressure formulation \eqref{eq::stokes}. The velocity $u$ and the pressure $p$ now belong, respectively,  to the spaces
\begin{align*}
  \Velspace &:= \vecb{H}_0(\divergence,\Omega),
  % = \{ \Velvar \in
              % \vecb{H}(\divergence,\Omega): \; \Velvar \cdot \normal = 0 \textrm{ on } \Gamma \}, 
  &&&
  \Presspace &:= L^2_0(\Omega):=\{ \Presvartest \in L^2(\Omega): \; \int_\Omega \Presvartest \dx =0 \}.
\end{align*}
%The skew symmetric part of $\nabla u$ denoted by $\gamma$ belongs to the space
%\begin{align*}
%W:= \{ \eta \in L^2(\Omega, {\mm}): \eta + \eta^\trans = 0\}.
%\end{align*}
Multiplying \eqref{eq::stokesintro_two-c} with a pressure test function $q \in Q$ and integrating over the domain $\Omega$ ends up in the familiar equation
$
(\divergence(u), q)_{L^2(\om)} =0,
$
which we write as the last equation of the final Stokes  system~\eqref{eq::mixedstressstokesweak} written below.
Here and throughout,  the inner product of a space $X$ is denoted by
$(\cdot, \cdot)_X$. When $X$ is the space of functions whose
components are square integrable functions on $\om$, we abbreviate $(
\cdot, \cdot)_X$ to simply $(\cdot, \cdot)$, as done
in~\eqref{eq::mixedstressstokesweak} below.
Similarly, while we  generally denote the norm and seminorm on a Sobolev space
$X$ by $\| \cdot \|_X$ and $|\cdot|_X$, respectively, 
to simplify notation, we set 
$\| f \|^2_D := (f, f)_D$,
where $(f, g)_D$ denotes
$L^2(D, \mathbb{V})$ inner product
for any
$\mathbb{V} \in \{ \rr, \rr^d, \kk, \mm\}$ and any 
subset $D \subseteq \Omega$. Moreover, when $D = \om$, we omit the
subscript and simply write $\| f \|$ for $\| f\|$.

To motivate the remaining equations of~\eqref{eq::mixedstressstokesweak}, let
the deviatoric part of a matrix $\sigma$ be defined by
$
  \Dev{\sigma} := \sigma - d^{-1}\trace{\sigma} \id,
$
where $\id$ denotes the identity matrix and $\trace{\sigma} := \sum_{i=1}^d \sigma_{ii}$ denotes the matrix trace. Since
$\nu^{-1} \sigma= \eps(u)$,
due to the incompressibility constraint $\divergence(u) = 0,$ we have the identity
\begin{align} \label{eq::devsigma}
\Dev{\nu^{-1} \Stressvar} =  \Dev{ \eps(\Velvar)} =  \eps(\Velvar) - \frac{\visc}{d} \trace{\eps(\Velvar)} \id = \eps(\Velvar) - \frac{1}{d} \divergence(\Velvar) \id =  \eps(\Velvar).
\end{align}
Since $\trace{\sigma}=0$ and $\sigma=\sigma^\trans$,
we define the stress space as the following
closed subspace of  $\Hcurldiv{\Omega}$: 
\begin{align*}
\Sigma^{\operatorname{sym}} := \{ \tau \in H(\curl \divergence, \Omega): \trace{\tau} = 0, \;\tau = \tau^\trans\}.
\end{align*}
Testing equations \eqref{eq::stokesintro_two-a} with a test functions $\tau \in \Sigma^{\operatorname{sym}}$ and integrating over the domain, we have for the term including $\eps(u)$ the identity
\begin{align*}
  \int_\Omega \eps(u) : \tau \dx &= \frac{1}{2} \int_\Omega \nabla u : \tau \dx + \frac{1}{2} \int_\Omega (\nabla u)^\trans : \tau \dx \\
  &= \frac{1}{2} \int_\Omega \nabla u : \tau \dx + \frac{1}{2} \int_\Omega \nabla u : \tau \dx = \int_\Omega \nabla u : \tau \dx.
\end{align*}
Using the knowledge that the velocity $u$ should be in $H_0^1(\om),$ we obtain
\begin{align*}
(\nu^{-1} \Dev{\Stressvar}, \Dev{\Stressvartest} ) 
    + 
    \ip{ \divergence(\Stressvartest),  \Velvar}_{H_0(\div, \om)}
    & = 0, 
\end{align*}
which is the first equation in the system~\eqref{eq::mixedstressstokesweak} below.  Here and throughout, when working with elements $f$ of the dual space $X^*$ of a topological space $X$, we denote the action of $f$ on an element $x \in X$ by $\ip{ f, x}_X$, where we may omit the subscript $X$ when its obvious from context. 
Finally we also test \eqref{eq::stokesintro_two-b} with $v \in V$ and integrate
the pressure term by parts. This results in the remaining equation of~\eqref{eq::mixedstressstokesweak}.

{\em Summarizing,} the weak problem is to find $(\sigma, u, p) \in \Sigma^{\operatorname{sym}} \times V  \times Q$ such that
\begin{align} \label{eq::mixedstressstokesweak}
  \left\{
  \begin{aligned}
    (\nu^{-1} \Dev{\Stressvar}, \Dev{\Stressvartest} ) 
    + 
    \ip{ \divergence(\Stressvartest),  \Velvar}_{H_0(\div, \om)} 
    & = 0 
    &&\text{ for all } \Stressvartest \in \Sigma^{\operatorname{sym}}, 
    \\
    \ip{\divergence(\Stressvar), \Velvartest}_{H_0(\div, \om)}
    +
    (\divergence(\Velvartest), \Presvar) 
    & = 
    -(\Forcevar, \Velvartest) %}_{H_0(\div, \om)} 
    &&\text{ for all } \Velvartest \in \Velspace,    
    \\
    (\divergence(\Velvar), \Presvartest) &=0 
    &&\text{ for all } \Presvar \in \Presspace.
\end{aligned}
\right.                                                         
\end{align}
In the ensuing section, we shall focus on a discrete analysis of a nonconforming scheme based on \eqref{eq::mixedstressstokesweak}. Although wellposedness of \eqref{eq::mixedstressstokesweak} is an interesting question, we shall not comment further on it here since it is of no direct use in a nonconforming analysis.

\section{The new method}

In \cite{gopledschoe}, we introduced an MCS method where $\sigma$ was an approximation to (the generally non-symmetric) $\nu \grad u$ instead of (the symmetric) $\nu \eps(u)$ considered above. Since there was no symmetry requirement in \cite{gopledschoe}, there we worked with the space $\Sigma := \{ \tau \in H(\curl \divergence, \Omega): \trace{\tau} = 0\}$ instead of $\Sigma^{\operatorname{sym}}$. The finite element space for $\Sigma$ designed there can be reutilized
in the current symmetric case
(with some modifications),
once we reformulate the symmetry requirement as a constraint in a weak form.

% thus without any symmetry constraint. The authors then defined a finite element space that was used to approximate functions in $\Sigma$. This discrete stress space was based on a finite element mapping that was designed such that normal-tangential continuity is preserved. Unfortunately, this mapping does not preserve symmetry, thus we can not choose a subspace of the discrete stress space defined in \cite{gopledschoe} which is capable to approximate functions just in $\Sigma^{\operatorname{sym}}$. However, as we still want to use the same finite element space as in \cite{gopledschoe} (motivated by Theorem \ref{th::normtangcont} below) we are going to incorporate symmetry in a weak sense. 

%The idea of a weakly enforcing symmetry is well known in the discretization of linear elasticity, see for example \cite{MR2336264, MR840802, MR2449101, MR1464150, Stenberg1988}. The basic idea is to enrich the stress space, which often consists of row wise $H(\divergence)$-conforming approximations, by appropriate ``$\curl$-bubbles'' which lie in the kernel of the divergence operator. Due to this enrichment, it is then possible to show discrete inf-sup stability. In this work we proceed similarly as discussed below. % enrich the discrete stress space $\Sigma_h$ by the bubbles introduced in \cite{MR2629995}.

To do so, we need further notation. Let $\kappa: \rr^{\dt} \to \kk$ be defined by 
\begin{equation}
  \label{eq:kappa}
  \kappa(v) =
  \frac 1 2 
  \begin{pmatrix}
    0 & -v \\
    v & 0         
  \end{pmatrix}
  \; \text{ if } d=2,
  \qquad\qquad%\qquad \text{ and }\qquad
  \kappa(v) =
  \frac 1 2 
  \begin{pmatrix}
    0 & -v_3 & v_2 \\
    v_3 & 0 & -v_1 \\
    -v_2 & v_1 & 0
  \end{pmatrix}
  \;\text{ if } d =3.  
\end{equation}
When $u$ represents the Stokes velocity, $\omega = \kappa(\curl(u))$ represents the {\em vorticity}. Since $\nabla u = \eps(u) + \omega$, introducing $\omega$ as a new variable, and the symmetry condition $\sigma - \sigma^\trans = 0$ as a new constraint, we obtain the boundary value problem 
\begin{subequations}
  \label{eq::mixedstressstokes}
  \begin{alignat}{2}
  \label{eq::mixedstressstokes-a}
  \frac{1}{\nu}{\Dev\Stressvar} - \nabla \Velvar + \omega & = 0 \quad && \textrm{in } \Omega,  \\
  \label{eq::mixedstressstokes-b}
 \divergence(  \Stressvar) - \nabla \Presvar & = -\Forcevar \quad && \textrm{in } \Omega,  \\
 \label{eq::mixedstressstokes-c} 
 \Stressvar - \Stressvar^\trans & = 0 \quad && \textrm{in } \Omega,  \\
  \label{eq::mixedstressstokes-d} 
  \divergence (\Velvar) &=0 \quad&& \textrm{in } \Omega, \\
  \label{eq::mixedstressstokes-e}
  \Velvar &= 0 \quad && \textrm{on } \Gamma.
 \end{alignat}
\end{subequations}
In the remainder of this section, we introduce a discrete formulation approximating~\eqref{eq::mixedstressstokes}.

The method will be described
on a subdivision (triangulation) $\mathcal{T}_h$ of $\Omega$
consisting of triangles in two dimensions and
tetrahedra in three dimensions.
For the analysis later, we shall assume that the $\mesh$ is quasiuniform.
By $h$ we denote the maximum of the diameters of all elements $T \in
\mathcal{T}_h$. Quasiuniformity implies that 
$  h \sim \operatorname{diam}(T)
$
for all mesh elements $T$.
Here and throughout, by $A \sim B$ we indicate that there exist two constants $c,C >0$ independent of the mesh size $h$ as well as the viscosity $\nu$ such
$cA \le B \le cA$. Similarly, we use the notation $A \lesssim B$ if there exists a constant $C \neq C(h,\nu)$ such that $A \le CB$.
All element interfaces and element boundaries on $\Gamma$ are called facets and are collected into a set
$\mathcal{F}_h$. This set is partitioned into facets on the boundary
$\facetsext$ and interior facets $ \facetsint$.  On each facet we
denote by $\jump{\cdot}$ the standard jump operator. On a boundary
facet the jump operator is just  the identity.  On all facets we
denote by $n$ a unit normal vector. When integrating over boundaries
of $d$-dimensional domains, the orientation of $n$ is assumed to be
outward. On
a facet with normal $n$ adjacent to 
an mesh element $T$,
the normal and tangential traces
of a smooth function $\phi: T \rightarrow \rr^d$
are defined by
$
\phi_n := \phi \cdot n
$ and
$\phi_t = \phi - \phi_n n,$ respectively.
Similarly, for a smooth $\psi: T \rightarrow {\mm}$, the (scalar-valued) ``normal-normal'' and the (vector-valued)  ``normal-tangential'' components are defined by 
$
  \psi_{nn} = \psi : (n \otimes n) = n^{\trans} \psi n
  $ and $\psi_{nt} = \psi n - \psi_{nn} n,$ respectively.

For any integers $m, k\ge 0$, the following ``broken spaces''  are viewed as
consisting of functions on $\om$ without any continuity constraints
across element interfaces:
\begin{align*}
  H^m(\mesh) % := \{ u \in L^2(\Omega): u|_T \in H^m(T) \textrm{ for all
  % } T \in \mesh \} 
  := \prod_{T \in \mesh} H^m(T),
  \qquad
  \Poly^k(\mesh) :=  \prod_{T \in \mesh} \Poly^k(T).
\end{align*}
For $D \subset \Omega$ we use the notation $(\cdot,\cdot)_{D}$ for the inner product of $L^2(D)$ or its vector and tensor analogues such as $L^2(D, \rr^d), L^2(D, \mm), L^2(D, \kk).$ Also 
let $\| \cdot \|^2_D = (\cdot ,\cdot)_D$. Next for each element $T \in \mesh$ let $\Poly^k(T)\equiv \Poly^k(T, \rr)$
denote the set of  polynomials of degree at most
$k$ on $T$. The vector and tensor analogues 
such as $
\Poly^k(T, \rr^d), \Poly^k(T, {\mm}),
\Poly^k(T, {\kk})$ have their components in
$\Poly^k(T)$.
The broken spaces $\Poly^k(\mesh, \rr^d), \Poly^k(\mesh, {\mm}),$ and
$\Poly^k(\mesh, {\kk})$ are
defined similarly. We shall also use
the conforming Raviart-Thomas space (see \cite{brezzi2012mixed, MR0483555}), 
$  \RT^k := \{ u_h \in H(\divergence, \om): u_h|_T \in \Poly^k(T, \rr^d) + x  {\Poly}^k(T, \rr) \textrm{ for all } T \in \mesh \}.$

\subsection{Velocity, pressure, and vorticity spaces}

For any $k \ge 1$, our method uses 
\begin{align*}
  V_h := V \cap \RT^k, \qquad
  Q_h := Q \cap \Poly^k(\mesh), \qquad
  W_h := \Poly^k(\mesh, \kk),
\end{align*}
for approximating the velocity, pressure, and vorticity, respectively.

Standard finite element mappings apply for these spaces. 
Let $\hat{T}$ be the unit simplex (for $d=2$ and $3$), which
we shall refer to as the {\em reference element}, and let 
$T \in \mesh$. Let $\phi: \hat{T} \rightarrow T$ be an affine
homeomorphism and set $F:= \phi'$. By quasiuniformity,
$ \| F\|_{\ell^\infty} \sim h,$
$ \| F^{-1}\|_{\ell^\infty} \sim h^{-1}, $ and $|\det{(F)}| \sim h^d,$
estimates that we shall use tacitly in our scaling arguments later.
Such arguments proceed by mapping functions on $\hat T$ to and from
$\hat T$. Given a scalar-valued $\hat{q}_h$,
a vector-valued $\hat{v_h}$,
and a skew-symmetric matrix-valued $\hat{\eta}_h$ on the reference
element $\hat T$, we map them to $T$ using
\begin{equation}
  \label{eq:map_PW}
  \mathcal{Q}(q_h) = \hat{q}_h \circ \phi^{-1},
  \quad
  \piola(\hat{v}_h):= \det(F)^{-1} F (\hat{v}_h\circ \phi^{-1}),
  \quad
  \mathcal{W}(\hat{\eta}_h) := F^{-\trans} (\hat{\eta}_h\circ \phi^{-1}) F^{-1},  
\end{equation}
respectively, i.e., these are our mappings for functions in the
pressure, velocity, and vorticity spaces, respectively. The first is
the inverse of the standard pullback, the second is the standard
Piola map,  and the third is designed to preserve skew symmetry.

\subsection{Stress space}

The definition of our stress space is motivated by 
the following result, proved in~\cite[Section 4]{gopledschoe}.

\begin{theorem} \label{th::normtangcont}
  Suppose  $\Stressvar$ is in $H^1(\mesh, {\mm})$ and
  $\Stressvar_{\normal\normal}|_{\d T} \in H^{1/2}(\partial T)$ for
  all elements $T \in \mesh$. Assume that the normal-tangential
  trace $\Stressvar_{nt}$ is continuous across element
  interfaces. Then $\sigma$ is in $\Hcurldiv{\Omega}$ and moreover
  \begin{equation} 
    \label{eq:10}
    \ip{ \div(\sigma), v}_{H_0(\div, \om)}
    = 
    \sum_{T \in \mesh}
    \left[
      (\div(\sigma), v)_T - \ip{ v_n, \sigma_{nn}}_{H^{1/2}(\d T)}
    \right]
  \end{equation}
  for all $v \in H_0(\div, \om).$
\end{theorem}

Clearly, matrix finite element subspaces having normal-tangential continuity are suggested by Theorem~\ref{th::normtangcont}. Technically, the theorem's sufficient conditions for full conformity  also include  the condition  $\Stressvar_{\normal\normal}|_{\d T} \in H^{1/2}(\partial T).$ This condition is very restrictive as it would enforce continuity at vertices and edges in two and three dimensions respectively. If this constraint is relaxed, much simpler, albeit nonconforming, elements can be constructed. This was the approach we adopted in
\cite{gopledschoe}. We continue in the same vein here
and define the nonconforming stress space 
\begin{align}
  \label{eq:Sigma_h}
\Stressspaceh &:= \{ \tau_h \in\Poly^k(\mesh, {\mm}):\;\trace{\tau_h} = 0,\;
                  \jump{(\tau_h)_{\normal\tangential}} =0 \text{ for all } F \in \mathcal{F}_h^{\textrm{int}}\}.
\end{align}
As mentioned in the introduction, we must enrich the above stress
space $\Stressspaceh$ to guarantee solvability of the resulting
discrete system due to the additional weak symmetry constraints. We
follow the approach of \cite{Stenberg1988} and its later
improvements~\cite{MR2629995,GopalGuzma12} to construct the needed
enrichment space.

Define a cubic matrix-valued ``bubble'' function as follows. On a
$d$-simplex $T$ with vertices $a_0, \ldots, a_d$, let $F_i$ denote the
face opposite to $a_i$, and let $\lambda_i$ denote the unique linear
function that vanishes on $F_i$ and equals one on $a_i$, i.e., the
$i$th barycentric coordinate of $T$. Following
\cite{MR2629995,GopalGuzma12}, we define $B \in \Poly^3(T, \mm)$ by
\begin{subequations}
  \label{eq::matrixbubble}
  \begin{align}
    \label{eq::matrixbubble:d3}
    B
    & =
      \sum\limits_{i=0}^3 \,\lambda_{i-3}\lambda_{i-2}\lambda_{i-1}\;
      \nabla \lambda_{i} \otimes \nabla \lambda_{i}
    && \text{ if } d=3,
    \\     \label{eq::matrixbubble:d2}
    B
    & = \lambda_0\lambda_1\lambda_2
    && \text{ if } d=2,
  \end{align}
\end{subequations}
where the indices on the barycentric coordinates are calculated mod~$4$ in
\eqref{eq::matrixbubble:d3}.
Let $\Poly_{\perp}^k(T, \vv)$ denote the $L^2$-orthogonal complement
of   $\Poly^{k-1}(T, \vv)$ in $\Poly^k(T, \vv)$ for  $\vv \in \{\rr,
\kk\}$, and let $\Poly_{\perp}^k(\mesh, \vv)= \prod_{T \in \mesh}
\Poly_{\perp}^k(T, \vv).$ For any $k\ge 1$, define 
\begin{align}
  \label{eq:deltaSigma}
  \delta \Sigma_h := \left\{\Dev{\curl (\curl(r_h) B)} : \;
  r_h \in \Poly_{ \perp}^k(\mesh, \kk)\right\},
\end{align}
for  $d=2$ and $3$, with the understanding that in $d=2$ case, the outer curl is defined by~\eqref{eq:curl-vec2mat-2d}, not~\eqref{eq:curl-vec2scal-2d}.
The total stress space is given by
\begin{align*}
  \Sigma_h^{+} := \Sigma_h \oplus \delta \Sigma_h, \qquad k\ge 1.
\end{align*}
That functions in this space have normal-tangential continuity is a consequence of the following property proved
in~\cite[Lemma~2.3]{MR2629995}.

\begin{lemma} \label{lem::bubbletransform}
   Let $q \in \mm$ and $T \in \mesh$.
   The products $qB$ and $Bq$ have vanishing tangential trace on $\partial T$,
  so  the function $\curl(q B)$ has vanishing normal trace on $\partial T$.
\end{lemma}

\begin{lemma}\label{lem::deltaspacenormtang}
  Any $\sigma \in \delta \Sigma_h$ has 
  vanishing $\sigma_{nt}$ and $\jump{ \sigma_{nt}} $ on all facets $F \in \facets.$
\end{lemma}
\begin{proof}
  Since $(\Dev{\sigma})_{nt} = \sigma_{nt}$, this is a direct consequence of Lemma~\ref{lem::bubbletransform}.
\end{proof}

We also need a proper mapping
for functions in $\Sigma_h^+$ that preserves normal-tangential
continuity. We shall continue to use the 
following map,  first introduced
in~\cite{gopledschoe}:
\begin{align}
  \label{eq:map_M}
  \mathcal{M}(\hat{\sigma}_h) := \frac{1}{\det(F)} F^{-\trans}
  (\hat{\sigma}_h\circ \phi^{-1}) F^\trans.
\end{align}
As shown in \cite[Lemma 5.3]{gopledschoe}, on each facet, 
$  (\mathcal{M}(\hat{\sigma}_h))_{nt}$ is a scalar multiple of $(\hat{\sigma}_h)_{nt}$
and 
$\trace{\hat{\sigma}_h} = 0$ if and only
if $\trace{\mathcal{M}(\hat{\sigma}_h)} = 0.$ Degrees of freedom are
discussed in \S\ref{ssec:dof}.

\begin{remark}
  \label{rem:implB}
  Note that in~\eqref{eq::matrixbubble},
  $B$ was given using barycentric coordinates as
  an expression that holds on any
  simplex. Let  $\hat{B}$ denote the function on  the reference
  element $\hat T$ obtained by 
  replacing $\lambda_i$ by
  reference element barycentric coordinates $\hat \lambda_i$.
  Considering the obvious  map that
  transforms $\hat\nabla\hat \lambda_i \otimes
  \hat\nabla\hat \lambda_i$ to
  $ \nabla\lambda_i \otimes
  \nabla\lambda_i$, we find that the matrix bubble $B$ on any simplex is given by 
  \begin{align} \label{eq:bubbletrafo}
    B := F^{-\trans } (\hat{B} \circ \phi^{-1}) F^{-1}. 
  \end{align}
\end{remark}

\subsection{Equations of the method}

For the derivation of the discrete variational formulation we turn our attention back to the weak formulation \eqref{eq::mixedstressstokesweak} and identify these forms:
\begin{align*}
  & \ablf: L^2(\om, \mm) \times L^2(\om, \mm)  \rightarrow \rr,
% (\Stressspaceh+\Stressspace) \times (\Stressspaceh+\Stressspace) \rightarrow \rr, 
  &&  \blfone: \Velspace \times \Presspace \rightarrow \rr,
  \\
  &\ablf(\Stressvar, \Stressvartest) := 
(\nu^{-1}     \Dev{\Stressvar}, \Dev{\Stressvartest}),    
  &&\blfone(\Velvar, \Presvar) := 
     (\divergence(\Velvar), \Presvar).
\end{align*}
The definition of the remaining bilinear form is motivated by the definition of the ``distributional divergence'' given by \eqref{eq:10}. To this end we define
$b_2 :
\{ \tau \in H^1(\mesh, \mm):
\jump{\tau_{nt}} =0\}  \times \left(
\{ v \in  H^1(\mesh, \RRR^d): \jump{v_n} =0\} \times L^2(\om, \mm) \right)\to \rr$
by 
\begin{align}
\blftwo(\tau, ( v,\eta)) &:=  \sum\limits_{T \in \mesh} \int_T
                   \divergence(\tau) \cdot v \dx  + \sum\limits_{T \in \mesh} \int_T
                   \tau : \eta  \dx - \sum\limits_{F \in
                   \facets} \int_F \jump{\tau_{nn}} v_n \ds. 
\label{eq::blftwoequione} 
\end{align}
Integrating the first integral by parts, we find the equivalent representation
\begin{align} 
  \blftwo(\tau, (v,\eta))
    &=  -\sum\limits_{T \in \mesh} \int_T \tau : (\nabla v - \eta) \dx + \sum\limits_{F \in \facets} \int_F \tau_{nt} \cdot \jump{v_t} \ds. \label{eq::blftwoequitwo}
\end{align}

Using these forms, we state the method.
For any $k \ge 1$,
the {\em discrete MCS method with weakly imposed symmetry}   finds $\sigma_h, u_h, \omega_h, p_h \in \Sigma_h^{+} \times V_h \times W_h \times Q_h$ such that
\begin{align} \label{eq::discrmixedstressstokesweak}
    \left\{
  \begin{aligned}
 \ablf (\Stressvarh ,\Stressvarhtest) + \blftwo(\Stressvarhtest, (\Velvarh, \omega_h)) & = 0 &&\text{ for all } \Stressvarhtest \in \Sigmaplus,  \\
\blftwo(\Stressvarh, (\Velvarhtest, \eta_h)) + \blfone(\Velvarhtest,
\Presvarh) & = (-\Forcevar, \Velvarhtest)  &&\text{ for all }
(\Velvarhtest, \eta_h) \in
U_h := \Velspaceh \times W_h,
\\
  \blfone(\Velvarh, \Presvarhtest) &=0 &&\text{ for all } \Presvarhtest \in \Presspaceh.
\end{aligned}
\right.    % \tag{MCS} % this tag makes it difficult to find the system
\end{align}
Since $V_h $ and $Q_h$ fulfills $\divergence(V_h) = Q_h$, the discrete
velocity solution component~$u_h$ satisfies $\divergence(u_h) = 0$
point wise, providing exact mass conservation.

% % Too much repetition of the same idea in this remark. Need a better
% % way to state this concisely just once, perhaps somewhere later.
%
% \begin{remark}
%   As mentioned in the introduction, in~\cite{gopledschoe}
%   we chose the BDM space of one higher order 
% $    \BDM^{k+1}:= H_0(\divergence, \om) \cap \Poly^{k+1}(\mesh, \rr^d)
% $
%   for the velocity approximations instead of $\RT^k.$
%   Local stabilizing stress bubbles were then needed to guarantee existence and uniqueness. If the velocity space is defined by Raviart-Thomas space $\RT^k$, no further stabilizing bubbles are needed, as we will show in Section~\ref{sec::apriori}. Due to the use of the
%   smaller space $\RT^k$, the convergence rate of
%   velocity error is reduced, but it can be increased to the optimal rate by a local post processing as we shall see in  Section~\ref{sec::postprocessing}. A similar choice of $H(\divergence)$-conforming space was made in a related 
%   hybridized DG method of~\cite{fujinqiu}.
% \end{remark}

\subsection{Degrees of freedom of the new stress space} \label{ssec:dof}

We need degrees of freedom (d.o.f.s) for the stress space that are
well-suited for imposing normal-tangential continuity across element
interfaces.  Since the bubbles in $\delta \Sigma_h$ have zero
normal-tangential continuity, we ignore them for this discussion and
focus on d.o.f.s that control $\Sigma_h$.

Consider
$\Sigma_T = \{ \tau|_T: \tau \in \Sigma_h\}$ on any mesh element
$T$. Letting $ \mathbb{D}$ denote the subspace of matrices
$M \in {\mm} $ satisfying $M : \id =0,$ we may identify $\Sigma_T$
with $\Poly^k(T, \mathbb{D})$. Let us recall a basis for $\mathbb{D}$
that was given in~\cite{gopledschoe}.
Define the following two sets of  constant matrix functions, 
for $d=2$ and $d=3$ cases, respectively, by
\begin{subequations}
    \label{eq:Sbasis}
\begin{gather}
  \label{eq:basisfacetwo}
  S^i := \operatorname{dev}\big(\nabla \lambda_{i+1} \otimes \curl(
  \lambda_{i+2})\big),
  \\  \label{eq:basisfacethree}
  S_0^i := \operatorname{dev}\!\big( \nabla \lambda_{i+1}
  \otimes (\nabla \lambda_{i+2}
          \times \nabla \lambda_{i+3})\big),   
          \quad S_1^i := \operatorname{dev}\!\big( \nabla
          \lambda_{i+2} \otimes (\nabla \lambda_{i+3} \times \nabla
          \lambda_{i+1})\big),
\end{gather}  
\end{subequations}
taking the indices mod~3 and mod~4, respectively.  We proved
in~\cite[Lemma~5.1]{gopledschoe} that the sets $\{ S^i : i=0,1,2\}$
and $\{ S^i_q : i=0,1,2,3, \; q =0,1 \}$ form a basis of $\dd$ when
$d=2$ and $3$, respectively.

Our  d.o.fs for $\Sigma_T \equiv \Poly^k(T, \mathbb{D}) $ are grouped
into two. The first group is associated to the set of element facets
($d-1$ subsimplices of $T$), namely, for each facet $F \in \partial T$,
we define the set of d.o.f.s
\begin{align*}
\Phi^F(\tau) := \int_F \tau_{nt} \cdot r \ds
\end{align*}
for each $r$ in any fixed basis for $\Poly^{k}(F, \rr^{d-1})$.  The
next group is the set of interior d.o.f.s, defined by
\begin{align*}
  \Phi^{0}(\tau) :=  \int_T \tau : \varsigma \dx
\end{align*}
for all $\varsigma$ in any basis of $\Poly^{k-1}(T, \mathbb{D})$.
We proceed to prove that the set of these d.o.f.s, 
$\Phi(T) := \Phi^{0}(\tau) \cup \{\Phi^F: F \subset \partial T \}$,
is unisolvent.

\begin{theorem} \label{theo::unisolvent}
  The set $\Phi(T)$ is a set of unisolvent d.o.f.s for
  $\Sigma_T \equiv \Poly^k(T, \mathbb{D})$.
\end{theorem}
\begin{proof}
  Suppose  $\tau \in \Sigma_T$ satisfies 
  $\phi(\tau) = 0$ for all d.o.f.s 
  $\phi \in \Phi(T)$. We need to show that $\tau=0$.
  From the facet d.o.f.s we conclude that $\tau_{nt}$ vanishes on
  $\d T$.  By~\cite[Lemma~5.2]{gopledschoe}, 
   $\tau$ may be expressed as
  \begin{align}
    \label{eq:tmpexp}
    \tau = \sum\limits_{i=0}^2
    \mu_i \lambda_i S^i \quad \textrm{or} \quad
    \tau = \sum\limits_{q=0}^1\sum\limits_{i=0}^3 \mu^q_i \lambda_i S_q^i,
  \end{align}
  when $d=2$ or $3$, respectively, where
  $\mu_i, \mu^0_i, \mu^1_i \in \Poly^{k-1}(T)$. The interior d.o.f.s
  imply that $\int_T \tau : s \dx =0$ for any $s \in
  \Poly^{k-1}(\hat{T}, \mathbb{D})$. Choosing for $s$
  the expression on the right hand
  side in~\eqref{eq:tmpexp} omitting the $\lambda_i$,
  say for the $d=2$ case, we  obtain 
  \begin{align*}
    \int_T \sum\limits_{i=0}^2 \mu_i \lambda_i S^i:
    \sum\limits_{i=0}^2 \mu_i S^i \dx
    =
    \int_T \lambda_i
    \bigg|\sum\limits_{i=0}^2 \mu_i S^i \bigg|^2 \dx = 0, 
  \end{align*}
  yielding $\mu_i = 0$, and thus $\tau=0$. A similar argument in $d=3$
  case yields the same conclusion that $\tau=0$. 

  To complete the proof, it now suffices to prove that
  $\dim(\Sigma_T) $ equals the number of d.o.f.s, i.e.,
  $\#\Phi(T)$. Obviously,
  $\dim(\Sigma_T) = \dim \Poly^k(T, \dd) = (d^2 -1) \dim
  \Poly^k(T)$. The cardinality of $\Phi(T)$ equals the sum of the number
  of facet d.o.f.s $(d+1) (d-1) \dim\Poly^k(T)$ and the number of interior
  d.o.f.s $(d^2-1) \dim \Poly^{k-1}(T)$, which simplifies to
  $(d^2-1) \big( \dim \Poly^{k-1}(T) + \dim \Poly^k(F)\big)$,
  equalling $\dim(\Sigma_T)$.
\end{proof}

Using these d.o.f.s, a canonical local interpolant $I_T(\tau)$ in
$\Sigma_T$ can be defined as usual, by requiring that
$ \psi(\tau - I_T\tau) =0,$ for all $ \psi \in \Phi(T).$

\begin{lemma}\label{lem:MapI}
  For any $\tau \in H^1(T, \mathbb{D}),$ we have
  $    \mathcal{M}^{-1}(I_T\tau) = I_{\hat{T}}(\mathcal{M}^{-1}(\tau)).$
\end{lemma}
\begin{proof}
  This proceeds along the same lines as the proof
  of~\cite[Lemma~5.4]{gopledschoe}.
\end{proof}

The global interpolant   $I_{\Sigma_h}$  is also defined as usual.
On each element $T \in \mesh$ the global
interpolant $(I_{\Sigma_h} \tau)|_T$ coincides with the local
interpolant $I_T(\tau|_T)$.

\begin{theorem} \label{the::sigmainterpolation}
  For any $m \ge 1$ and any  $\sigma \in \{ \tau \in H^m(\mesh,
  \mathbb{D}): \jump{\tau_{nt}} = 0\}$, the global interpolation
  operator $I_{\Sigma_h}$  satisfies
  \begin{align*}
    \| \sigma - I_{\Sigma_h}\sigma \|^2 +
    \sum\limits_{F \in \facets} h \| (\sigma-I_{\Sigma_h}\sigma)_{nt}
    \|_F^2
    \;\lesssim\;
    h^{2s} \| \sigma\|^2_{H^s(\mesh)},
  \end{align*}
  for all $s \le \min(k+1,m)$.
\end{theorem}
\begin{proof}
  This follows from a standard Bramble-Hilbert argument using
  Lemma~\ref{lem:MapI}.
\end{proof}

\section{A priori  error analysis} \label{sec::apriori}

In this section we first show the stability of the MCS method with
weakly imposed symmetry by proving a discrete inf-sup condition
(Theorem~\ref{theorem::LBB}).  We then prove consistency
(Theorem~\ref{the::consistency}), optimal error estimates
(Theorem~\ref{the::errorconvergence}), and pressure robustness
(Theorem~\ref{the::presrobust}). For simplicity, the analysis from now
on assumes that $\nu$ is a constant.

\subsection{Norms} \label{sec:norms}

In addition to the previous notation for norms (established in
Section~\ref{sec:preliminaries}), hereon we also use $ \| \cdot \|_h^2$
to abbreviate $\sum_{T \in \mesh} \| \cdot \|_T^2$, a notation that
also serves to indicate that certain seminorms are defined using
differential operators applied element by element, not globally, e.g.,
\begin{gather*}
  \| \eps(v)\|_h^2 := \sum_{T \in \mathcal{T}_h} \| \eps(v) \|_T^2, \qquad
  \| \curl(\gamma)\|_h^2 := \sum_{T \in \mathcal{T}_h} \| \curl(\gamma) \|_T^2,
  \\
  \|{ v }\|_{1,h,\eps}^2  := 
  % \sum\limits_{T \in \mesh}  \|  \eps(v_h) \|_{T}^2
  \|  \eps(v) \|_h^2
  + \sum\limits_{F \in \facets} 
  \frac{1}{h} \big\| \jump{{ v_\tangential}} \big\|^2_{F},  
\end{gather*}
for $v \in H^1(\mesh, \rr^d)$ and $\gamma \in H^1(\mesh, \mm)$.
Recall that $U_h = V_h \times W_h$.  Our analysis is based on norms of
the type used in~\cite{Stenberg1988}. Accordingly, we will need to use
the following norms for $v_h \in V_h$ and $\eta_h \in W_h$:
\begin{gather*}
  \| v_h \|_{V_h}^2
  = \|{ v_h }\|_{1,h,\eps}^2,  
  \qquad
  \| (v_h, \eta_h) \|_{U_h}^2
  := \|{ v_h }\|_{1,h,\eps}^2 +
  \|  \kappa(\curl v_h) - \eta_h \|_h^2.
\end{gather*}
Lemma~\ref{lem::singlenorms} below will show that
the latter is indeed   a norm.
% While both $\|\cdot \|_{V_h}$ and $\| \cdot \|_{U_h}$ include a
% discrete version of an $H^1$-like seminorm obtained using $\eps(v_h)$,
% the second also measures the differences between the velocity $v_h$
% and some given vorticity candidate $\eta_h$.  

On the discrete space
$U_h$, we will also need another norm defined using the following projections. On
any mesh element $T$, let $\Pi^{k-1}_T$ denote the $L^2(T, \vv)$
orthogonal projection onto $\Poly^k(T, \vv)$ where $\vv$ is determined
from context to be an appropriate vector space such as $\rr^d,$ or
$ \mm$. When the element $T$ is
clear from context, we shall drop the subscript $T$ in $\Pi^{k-1}_T$
and simply write $\Pi^{k-1}$.
Also, on each facet $F \in \facets,$  we introduce a
projection onto the tangent plane $n_F^\perp$: for any
$v \in L^2(F, n_F^\perp)$, the projection $\proj_F^1 v \in
\Poly^1(F, n_F^\perp)$ is defined by 
$(\proj_F^1 v, r)_F = (v, r)_F$ for all $r \in \Poly^1(F, n_F^\perp)$.
Using these, define
\begin{gather}
  \label{eq:Uh*norm}
  \| (v_h, \eta_h) \|_{U_h,*}^2  := 
  \sum\limits_{T \in \mesh} \| \Pi^{k-1}_T \Dev{\nabla v_h - \eta_h} \|_{T}^2
  + \sum\limits_{F \in \facets} 
  \frac{1}{h} \| \Pi^1_F\jump{{ (v_h)_\tangential}} \|^2_{F}.
\end{gather}
Lemma~\ref{lem::normequi} below will help us go between this norm and $\|
(v_h, \eta_h)\|_{U_h}$.
% When $\|\cdot \|_{U_h,*}$ is applied to the discrete solution
% component $(u_h, \omega_h) \in U_h$, then the result should be close
% to $\| (u_h, \omega_h)\|_{U_h}$: indeed, we expect
% $\kappa(\curl u_h) $ to approximate $\omega_h,$ and moreover,
% $\divergence(u_h) = 0$, so the first term in the norm is approximately
% a norm of $\Pi^{k-1} \eps(u_h)$.

The remaining spaces $ \Sigmaplus$ and $Q_h$ are simply normed by the
$L^2$ norm $\| \cdot \|$. The full
discrete space is normed by
\begin{equation}
  \label{eq:norm1}
  \|(v_h, \eta_h,\tau_h, q_h) \|_* :=
  \sqrt{\visc} \Velnormh{ (v_h, \eta_h) } +
  \frac{1}{\sqrt{\visc}} ( \Sigmanormh{\tau_h} +
  \Presnormh{q_h})  
\end{equation}
for any $(v_h, \eta_h,\tau_h, q_h)  \in V_h \times W_h \times
\Sigma_h^+ \times Q_h$.

\subsection{Norm equivalences} \label{ssec:equivalences}

Next, we use the finite element mappings introduced earlier
--see~\eqref{eq:map_PW} and~\eqref{eq:map_M}-- to  show several norm equivalences.

\begin{lemma} \label{lem:equiv1}
  Let $\tau_h \in \Sigma_h^+$. Then 
\begin{align}  \label{eq::sigmascaling}
  h^d \| \tau_h \|^2_T &\sim \| \hat{\tau}_h \|^2_{\hat{T}} \quad \textrm{for all} \quad T \in \mesh \\
  h^{d+1} \| (\tau_h)_{nt} \|^2_F &\sim \| (\hat{\tau}_h)_{\hat n \hat t} \|^2_{\hat{F}} \quad \textrm{for all} \quad F \in \facets.
  \\
  \label{eq::sigmanormequi}
  \Sigmanormh{ \tau_h }^2
  & \sim \sum\limits_{T \in \mesh} \| {\tau_h}
    \|_{T}^2 + \sum\limits_{F \in \facets} h 
    \big\| \jump{(\tau_h)_{\normal\tangential}} \big\|_{F}^2.
\end{align}
\end{lemma}
\begin{proof}
The first two  follow by a simple scaling argument. For the third, see the proof of \cite[Lemma 6.1]{gopledschoe}.
\end{proof}

In the proof of the next lemma, we use the space of rigid displacements $\RM = \Poly^0(T, \rr^d) + \Poly^0(T, \kk)\, x.$ For each element $T \in \mesh$, let $\Pi^{{\RM}}:H^1(T) \to \RM$ denote the projector defined in~\cite{MR2047078}.
Then, for any 
 $v_h \in V_h,$ the projection
$\Pi^{{\RM}} v_h \in \RM$ % is defined by 
% \begin{align*}
%   \int_T \Pi^{{\RM}}v_h \cdot  r \dx =   \int_T v_h \cdot r \dx, \qquad r \in \RM.
% \end{align*}
% Note, that this projection
fulfills the properties (see  \cite[eq.~(3.3), (3.11)]{MR2047078})
\begin{alignat}{2}
  \| \nabla (v_h -  \Pi^{{\RM}}v_h) \|_T &\sim \| \eps ( v_h)\|_T   
  && \quad \text{ for all } T \in \mesh, \label{eq::kornone}\\
  \big\| \jump{v_h -  \Pi^{{\RM}}v_h} \big\|_F^2 &\lesssim \sum\limits_{T: T \cap F \neq \emptyset} h\| \eps ( v_h)\|^2_T
  && \quad \text{ for all } F \in \facets. \label{eq::korntwo}
\end{alignat}
We shall also use a global discrete Korn inequality, implied by 
\cite[Theorem~3.1]{MR2047078}. Namely, there is an $h$-independent constant
$\ckorn$ such that 
\begin{equation}
  \label{eq:discrete_Korn}
  \ckorn^2 \|\grad v\|_h^2 \le 
  \| \eps(v) \|_h^2
  +
  \sum_{F \in \facets} h^{-1} \big\| \Pi_F^1 \jump{v} \big\|_F^2,
  \quad \text{ for all } v \in H^1(\mesh, \rr^d).
\end{equation}

\begin{lemma} \label{lem::normeqvel} 
  For all $(v_h, \eta_h) \in U_h$,
  \begin{align*}
    \| (v_h, \eta_h) \|_{U_h}^2 \sim
    \|  \eps(v_h) \|_h^2 + \|  \kappa(\curl v_h) - \eta_h \|_h^2
    + \sum\limits_{F \in \facets} 
    \frac{1}{h} \big\| \proj_F^1\jump{{ (v_h)_\tangential}} \big\|^2_{F}    
  \end{align*}
\end{lemma}
\begin{proof}
  One side of the equivalence is obvious by the continuity of the $\proj_F^1$. For the other direction first note that
  $h^{-1}\| \jump{{ (v_h)_\tangential}} \|^2_{F} \le 2 h^{-1}
  \| \proj_F^1  \jump{(v_h)_\tangential} \|^2_{F} +
  2 h^{-1}\| \jump{{ (v_h - \proj_F^1 v_h)_\tangential}} \|^2_{F}.$
  As $\Pi^{{\RM}}v_h \in \Poly^1(T, \rr^d)$ we have again by the continuity of $\proj_F^1,$
  \begin{align*}
    \| \jump{{ (v_h - \proj_F^1 v_h)_\tangential}} \|^2_{F} = \| (\id - \proj_F^1) \jump{{ (v_h - \Pi^{{\RM}}v_h)_\tangential}} \|^2_{F} \le \| \jump{{ (v_h - \Pi^{{\RM}}v_h)_\tangential}} \|^2_{F}.
  \end{align*}
  We conclude the proof using \eqref{eq::korntwo}.
\end{proof}

The following well-known  property of Raviart-Thomas spaces (see, e.g., \cite[Lemma~3.1]{CockbGopal04})  is needed at several points.

\begin{lemma}\label{lem::divfreeRT}
  Let $v \in \Poly^k(T, \rr^d) +  x\Poly^k(T, \rr)$ and $\divergence(v)=0$. Then $v$ is in $\Poly^k(T, \RRR^d)$.
  % and has the local representation
  % \begin{align*}
  %   v_h|_T = a_T \quad \textrm{with} \quad a_T \in \Poly^k(T,\rr^d).
  % \end{align*}
\end{lemma}

\begin{lemma} \label{lem:gradofRT}
  For all $T \in \mesh$ and $v_h \in V_h,$ 
  \begin{gather}
    \label{eq:eps-equiv-1}
    \| \eps(v_h)\|_T^2
    \sim \| \Pi^{k-1}\Dev{\eps(v_h)} \|_T^2 + \| \divergence (v_h) \|_T^2
    % \\
    %                    &= \| \Pi^{k-1}[\Dev{\nabla(v_h)} - \kappa(\curl(v_h))] \|_T^2 + \| \divergence (v_h) \|_T^2,      
    \\
    \label{eq:(I-Pi)curl}
    \| (\id - \Pi^{k-1}) \kappa(\curl v_h) \|_T^2
    \lesssim \| \divergence (v_h) \|_T^2,
    \\
    \label{eq:(I-Pi)grad}
    \| (\id - \Pi^{k-1}) \nabla v_h \|_T^2
    \lesssim \| \divergence (v_h) \|_T^2.
 \end{gather}
 % and the estimate
 %  \begin{align*}
 %    \| (\id - \Pi^{k-1}) \kappa(\curl v_h) \|_T^2 &\lesssim \| \divergence (v_h) \|_T^2, \quad \textrm{and} \quad \| (\id - \Pi^{k-1}) \nabla v_h \|_T^2 \lesssim \| \divergence (v_h) \|_T^2.
 %  \end{align*}
\end{lemma}
\begin{proof}
  One side of the equivalence of~\eqref{eq:eps-equiv-1} is obvious by the continuity of the $\Pi^{k-1}$. For the other direction, we use the following equivalence on the reference element $\hat T$:
  \begin{equation}
    \label{eq:ref-Euler}
    \|\hat{\nabla}(\hat{q} \hat{x}) \|_{\hat T}
    \sim
    \| \hat{\divergence}( \hat{q}\hat{x}) \|_{\hat T},
    \qquad \text{ for all } \hat q \in \Poly^k(\hat T, \rr).
  \end{equation}
  This follows by  finite dimensionality,  because by the Euler identity if any one of the above two terms is zero, then $\hat q =0$ (see e.g., \cite{2016arXiv160903701L}).
  Consequently, given any $v_h \in V_h$,
  setting  $\hat{v}_h = \mathcal{P}^{-1} (v_h|_T)$, the
  following problem is uniquely solvable: find $\hat b \in \Poly^k(\hat T , \rr)$ such that
  \begin{align} \label{eq:divproblem}
    \int_{\hat T}
    \hat{\divergence}
    ( \hat x \hat b)
    \;
    \hat{\divergence}
    ( \hat x \hat q) \dx
    =
    \int_{\hat T} \
    \hat{\divergence}(\hat{v}_h)
    \;
    \hat{\divergence}( &\hat x \hat q) \dx,
    &\quad \text{ for all } \hat q \in \Poly^k(\hat T, \rr). 
  \end{align}
  Since $\hat{\divergence} (\hat x \Poly^k(\hat T, \rr)) = \Poly^k(\hat T, \rr)$,
  \eqref{eq:divproblem} implies that
  $ \hat{\divergence}
  ( \hat x \hat b) = \hat{\divergence} (\hat v_h)$.
  Put $r = \mathcal{P}^{-1}( \hat x \hat b)$. 
  Then, due to the properties of the Piola map $\mathcal{P}$,
  $r$ is a function in  $\Poly^k(T, \rr^d) + x \Poly^k(T, \rr)$ satisfying
  $\divergence (r) = \divergence (v_h)$ in $T$,
  and a scaling argument using~\eqref{eq:ref-Euler} implies
  \begin{equation}
    \label{eq:phy-Euler}
    \| \nabla r\|_{T}
    \sim
    \| {\divergence}( r) \|_{T}.
  \end{equation}

  Let $a = v_h-r \in \Poly^k(T, \rr^d) + x \Poly^k(T, \rr)$. Then $\divergence(a)=0$ and $v_h = a + r$ in $T$. In particular, the  former implies, by Lemma~\ref{lem::divfreeRT}, that $a \in \Poly^k(T, \RRR^d)$.
  Then we have
  \begin{align*}
    \| \eps(v_h)\|_T
    & = \|\eps(a + r) \|_T
    \lesssim \|\Dev{\eps(a+r)}\|_T + \|\divergence(v_h)\|_T
    \\
    & \le \|\Dev{\eps(a)}\|_T + \|\nabla r\|_T + \|\divergence(v_h)\|_T  \\
    & \lesssim \|\Dev{\eps(a)}\|_T + \|\divergence(v_h)\|_T
    && \text{ by~\eqref{eq:phy-Euler},}
    \\
    & = \|\Pi^{k-1}\Dev{\eps(a)}\|_T + \|\divergence(v_h)\|_T
    && \text{ since } a \in \Poly^k(T, \RRR^d),
    \\
    & \le \|\Pi^{k-1}\Dev{\eps(v_h)}\|_T
      + \|\Pi^{k-1}\Dev{\eps(r)}\|_T
      + \|\divergence(v_h)\|_T
    \\
    & \lesssim \|\Pi^{k-1}\Dev{\eps(v_h)}\|_T      
      + \|\divergence(v_h)\|_T,
    && \text{ again, by~\eqref{eq:phy-Euler}}.
  \end{align*}                                          
This proves~\eqref{eq:eps-equiv-1}.
    
To prove~\eqref{eq:(I-Pi)curl},
first note that due to the definition of $\kappa(\cdot),$ we have $\| \kappa(\curl v_h)\|_T \sim \| \curl(v_h) \|_T$. Thus, using the same decomposition as above,
namely, $v_h|_T = a + r$, 
  \begin{align*}
    \| (\id - \Pi^{k-1}) \kappa(\curl (v_h)) \|_T
    \le \| (\id - \Pi^{k-1})  \kappa (\curl(a)) \|_T
    +
    \| (\id - \Pi^{k-1}) \kappa (\curl(r)) \|_T.
  \end{align*}
  As $\curl (a) \in \Poly^{k-1}(T, \rr^\dt)$, the first term on the
  right  vanishes.
  The last term satisfies
  \begin{align*}
    \| (\id - \Pi^{k-1}) \kappa(\curl(r)) \|_T
    &\lesssim \| \curl(r) \|_T  \le \| \grad r \|_T  \lesssim
      \|\divergence(r)\|_T = \|\divergence(v_h)\|_T,
  \end{align*}
  due to~\eqref{eq:phy-Euler}. Hence~\eqref{eq:(I-Pi)curl} is proved.

  The proof of~\eqref{eq:(I-Pi)grad} uses the same technique:
  \[
    \| (\id - \Pi^{k-1}) \nabla v_h\|_T
    \le
    \| (\id - \Pi^{k-1}) \nabla a\|_T +     \| (\id - \Pi^{k-1}) \nabla r\|_T
    \le \| \nabla r\|_T \lesssim \| \divergence (v_h)\|_T,
  \]
  where we have used that 
  $a \in \Poly^k(T, \rr^d)$ and~\eqref{eq:phy-Euler}.
\end{proof}

\begin{remark}
  \label{rem:BDM1}
  The same technique shows that
  $ \| \nabla v_h\|_T^2 \sim \| \Pi^{k-1}[\Dev{\nabla v_h}] \|_T^2
  + \| \divergence (v_h) \|_T^2$
  for all Raviart-Thomas functions $v_h \in V_h$. The technique allows
  one to control the gradient of the highest order terms of a
  velocity $v_h$ in the Raviart-Thomas space
  by $\divergence(v_h)$. A similar estimate does {\em not}
  hold for $v_h$
  in
  $\BDM^{k+1}\!:= H_0(\divergence, \om) \cap \Poly^{k+1}(\mesh, \rr^d).$
\end{remark}

\begin{lemma} \label{lem::gradofSKW}
  For all $T \in \mesh$ and $\eta_h \in W_h$,
  \begin{align*}
    \| \grad \eta_h\|_T \sim \| \curl \eta_h \|_T.
  \end{align*}
\end{lemma}
\begin{proof}
  The proof is based on a scaling argument and equivalence of norms on
  finite dimensional spaces on the reference element. Recall the map
  $\phi$ and $F = \phi'$. Calculations using the chain rule yield
  \begin{subequations}
    \label{eq:curlFFt}
    \begin{align}
      \label{eq:curlFFt-d3}
      \hatcurl \big[ F^\trans  (\eta_h\circ \phi ) F \big]
      &= F^\trans \big[ \curl( \eta_h) \circ \phi\big] F^{-\trans}
      \det F,
      && \text{ if } d =3,
      \\
      \label{eq:curlFFt-d2}
      \hatcurl \big[ F^\trans  (\eta_h\circ \phi ) F \big]
      & = F^\trans \big[ \curl( \eta_h) \circ \phi\big]
      \det F,
      && \text{ if } d =2.
    \end{align}
  \end{subequations}
  We continue with the $d=3$ case only (since $d=2$ case proceeds using
  \eqref{eq:curlFFt-d2} analogously).
  With $\hat{\eta}_h = F^\trans (\eta_h \circ \phi) F$, standard
  estimates for $F$ yield
  \begin{align}
    \label{eq:curl-map-bd}
    %  What kind of tensor is \nabla \eta_h ???
    % \| \nabla \eta_h \|^2_T &\sim h^{d-6} \| \hat{\nabla} \hat{\eta}_h \|^2_{\hat{T}} \quad \textrm{and} \quad
                                \| \curl( \eta_h) \|^2_T \sim h^{-3} \| \hatcurl (\hat{\eta}_h) \|^2_{\hat{T}}.
%    \| \curl \eta_h \|^2_T &\sim h^{-4} \| \hatcurl \hat{\eta}_h \|^2_{\hat{T}} \quad \textrm{ for } d=2.
  \end{align}
  Let $\hat v \in \Poly^k(\hat T, \rr^d)$
  and
  $v \in \Poly^k(T, \rr^d)$ be such that $\hat{\eta}_h = \kappa(\hat v)$
  and
   $\eta_h = \kappa(v)$,
   where $\kappa$ is as defined in~\eqref{eq:kappa}.
   Then, 
   \begin{align}
     \label{eq:grad-map-bd}
     \|\grad \eta_h\|_T^2 \sim     \| \nabla v \|^2_T
     &\sim
     % h^{d-6}
       h^{-3}
       \| \hat{\nabla} \hat{v} \|^2_{\hat{T}}
       \sim
       h^{-3} \|\hat{\grad}\hat \eta_h \|_{\hat T}^2
   \end{align}
   
   In view of~\eqref{eq:curl-map-bd} and~\eqref{eq:grad-map-bd},
   to complete the proof, it suffices to establish the reference element estimate 
   \begin{equation}
     \label{eq:ref-elem-grad-curl}
     \| \hatcurl (\kappa(\hat{v})) \|_{\hat{T}}
     \sim \| \hat{\nabla} \hat{v} \|_{\hat{T}}
   \end{equation}
   by proving that one side is zero if and only if  the other side is zero.
   Note these two identities:
   $\hatcurl\, \kappa(\hat{v})
   = (\hat\nabla \hat{v})^\trans - \hat\divergence(\hat{v}) \id,$ and
   $\hatcurl \kappa(\hat{v}): \id = -2 \divergence (\hat{v})$. 
   If  $\hatcurl\, \kappa(\hat{v}) = 0$, then the latter identity implies
   $\hat\divergence (\hat v) =0$, which when  used in the former identity, yields
   $\hat\nabla \hat v=0$. Combined with the obvious converse, we have established~\eqref{eq:ref-elem-grad-curl}.
\end{proof}

\begin{lemma} \label{lem::normequi}
  For all $T \in \mesh$ and $(v_h,\eta_h) \in U_h$,
  \begin{align*}
    \|  \eps(v_h) \|_{T}^2
    + \|  \kappa(\curl v_h) - \eta_h \|_{T}^2  
    \;\sim\;  \|  \Pi^{k-1}\Dev{\nabla v_h - \eta_h}\|_{T}^2 +
    h^2 \|  \curl(\eta_h) \|_{T}^2 +  \| \divergence(v_h) \|_T^2.
  \end{align*}
\end{lemma}
\begin{proof}
  Since the decomposition $\nabla v_h = \eps(v_h) +
  \kappa(\curl(v_h))$ is orthogonal in the Frobenius inner product, so
  is  $\nabla v_h - \eta_h = \eps(v_h) +
  [\kappa(\curl(v_h) - \eta_h].$ Application of the deviatoric and 
  $\Pi^{k-1}$ preserves this orthogonality.  Hence, by Pythagoras
  theorem,
  \begin{align} \label{eq::pythagoras}
    \big\|  \Pi^{k-1}\Dev{\nabla v_h - \eta_h}\big\|_{T}^2
    = \big\|  \Pi^{k-1}\Dev{ \eps(v_h)}\big\|_{T}^2
    + \big\| \Pi^{k-1}[\kappa(\curl(v_h)) - \eta_h]\big\|_{T}^2.
  \end{align}
  We shall now prove the result using~\eqref{eq::pythagoras} and
  Lemma~\ref{lem:gradofRT}.

  Proof of ``$\lesssim$'':  Since
  \begin{align*}
    \| \eps(v_h)\|_T^2
    & \lesssim \| \Pi^{k-1}\Dev{\eps(v_h)} \|_T^2 + \| \divergence (v_h) \|_T^2
    && \text{by Lemma~\ref{lem:gradofRT}},
    \\
    &\le \big\| \Pi^{k-1}\Dev{\nabla v_h      - \eta_h}
      \big\|_T^2
      + \| \divergence (v_h) \|_T^2
    && \text{by \eqref{eq::pythagoras}},
  \end{align*}
  it suffices to prove that
  \begin{equation}
    \label{eq:kappav-eta}
    \|  \kappa(\curl(v_h)) - \eta_h \|^2_{T}
    \lesssim  \big\| \Pi^{k-1}\Dev{\nabla v_h - \eta_h}
    + h^2 \| \curl(\eta_h) \|_T^2 + \| \divergence (v_h)\|_T^2,
  \end{equation}
  which we do next. Since the projection $r_1 =
   \Pi^{k-1}( \kappa(\curl(v_h))  - \eta_h)$ can be bounded
   using~\eqref{eq::pythagoras}, we focus on the remainder
   $r_2 = (\id - \Pi^{k-1}) (\kappa(\curl(v_h)) - \eta_h)$.
   \begin{align*}
     \| r_2\|_T^2
     & \le
       \|  (\id - \Pi^{k-1}) \kappa(\curl(v_h)) \|^2_{T} + \|  (\id -
       \Pi^{k-1}) \eta_h \|^2_{T}
     \\
     &\le
       \|   \divergence(v_h) \|^2_{T} + h^2  \|\grad \eta_h \|_T^2
     && \text{by~\eqref{eq:(I-Pi)curl}, Lemma~\ref{lem:gradofRT},}
     \\
     &\lesssim
       \|   \divergence(v_h) \|^2_{T} + h^2 \|  \curl(\eta_h)
       \|^2_{T}
       && \text{by Lemma~\ref{lem::gradofSKW}.}
   \end{align*}
   When this estimate for $r_2$ is used in 
   $\|  \kappa(\curl(v_h))  - \eta_h \|^2_{T}
   = \| r_1\|_T^2
    + \|  r_2 \|_T^2$ and $r_1$ is bounded
    using~\eqref{eq::pythagoras}, we obtain~\eqref{eq:kappav-eta}.

    Proof of ``$\gtrsim$'':  The last term of the lemma obviously satisfies
    $\| \divergence(v_h) \|_T^2 \lesssim \| \eps(v_h)\|_T^2$, while
    the first term satisfies 
    $    \|  \Pi^{k-1}\Dev{\nabla v_h - \eta_h}\|_{T}^2
    \le \|  \eps(v_h)\|_{T}^2 + \| \kappa(\curl(v_h)) - \eta_h\|_{T}^2
    $
    by \eqref{eq::pythagoras}. It remains to bound
    $h^2 \| \curl(\eta_h)\|_T^2$.
    As $\curl[\kappa(\curl(\Pi^{\RM}v_h))] = 0$, we obtain using an inverse inequality for polynomials 
  \begin{align*}
    h^2 \|  \curl \eta_h \|_{T}^2
    &=     h^2 \|  \curl( \eta_h - \kappa(\curl( \Pi^{\RM}v_h))) \|_{T}^2 
      \lesssim \|  \eta_h - \kappa(\curl\Pi^{\RM}v_h) \|_{T}^2
    \\
    &\le \|  \eta_h - \kappa(\curl(v_h))\|^2_T + \| \kappa(\curl(v_h)) - \kappa(\curl\Pi^{\RM}v_h) \|_{T}^2 \\
                                  &\sim \|  \eta_h -
                                    \kappa(\curl(v_h))\|^2_T + \|
                                    \curl(v_h- \Pi^{\RM}v_h) \|_{T}^2 \\
                                  &\lesssim \|  \eta_h - \kappa(\curl(v_h))\|^2_T + \| \eps(v_h) \|_T^2,
  \end{align*}
  where we used~\eqref{eq::kornone} in the last step.
\end{proof}

\begin{lemma} \label{lem::singlenorms}
  For any $v_h \in V_h$ and $\gamma_h \in W_h,$
  \begin{align}
    \label{eq:singlenorms-1}
    h\|\grad \gamma_h\|_h^2
    &\lesssim \inf\limits_{v_h \in V_h} \| (v_h, \gamma_h) \|_{U_h}
      \le
      \|\gamma_h\|^2,
    &
    \| v_h\|_{1,h,\eps} &= \inf\limits_{\eta_h \in W_h} \| (v_h, \eta_h) \|_{U_h}.
  \end{align}
  While the first estimate in~\eqref{eq:singlenorms-1} involves only
  the local constants from Lemmas~\ref{lem::gradofSKW}
  and~\ref{lem::normequi}, using the global constant $\ckorn,$ we also
  have
  \begin{align}
    \label{eq:singlenorms-2}
    (1+ \ckorn)^{-1}\|\gamma_h\|_h
    &\le \inf\limits_{v_h \in V_h} \| (v_h, \gamma_h) \|_{U_h}.
  \end{align}
\end{lemma}
\begin{proof}
  To prove the first estimate of~\eqref{eq:singlenorms-1},
  \begin{align*}
    \| (v_h, \gamma_h) \|
    & \ge \| \eps(v_h) \|_h^2 + \| \kappa(\curl v_h) - \gamma_h\|_h^2
      \gtrsim h^2 \| \curl \gamma_h\|_h^2
    && \text{ by Lemma~\ref{lem::normequi}}
    \\
    & \gtrsim h^2 \| \grad \gamma_h \|_h^2
      && \text{ by Lemma~\ref{lem::gradofSKW}.}
  \end{align*}
  Taking infimum over $v_h \in V_h$, we obtain the lower estimate
  of~\eqref{eq:singlenorms-1}. The upper bound of the first infimum
  obviously follows by choosing $v_h=0$.

  To prove the equality in~\eqref{eq:singlenorms-1}, observe that the
  infimum over $\eta_h \in W_h$ cannot be larger than
  $\| v_h \|_{1,h,\eps}$ because we may choose
  $\eta_h = \kappa(\curl v_h)$.  The reverse inequality also holds
  since $\| (v_h, \eta_h) \|_{U_h} \ge \| v_h\|_{1,h,\eps}$ for any
  $\eta_h\in W_h$, so the equality must hold.

  Finally, to prove~\eqref{eq:singlenorms-2}, we use triangle
  inequality to get
  \[
    \|\eta_h\|
    \le \| \kappa(\curl v_h) - \eta_h \|_h + \| \curl    v_h\|_h
    \le \| (v_h, \eta_h)\|_{U_h} + \|\grad v_h \|_h.
  \]
  Applying the Korn inequality~\eqref{eq:discrete_Korn} and noting that
  the jump of the normal components are zero for functions in 
  $v_h \in H_0(\div, \om)$, the proof is complete.
\end{proof}

\subsection{Stability analysis}

The next three lemmas lead us to a discrete inf-sup condition. 

\begin{lemma}
  \label{lem:bubble-equiv}
  Let $\mu \in \Poly^k(T, \mm)$ for some $T\in \mathcal{T}_h$
  and $\tau = (\det F) \Dev{\curl(\curl (\mu) B)}$.  Then for $d=3,
  2$, 
  \[
    \| \tau \|_T\sim h^{3-d}\| \curl (\mu)\|_T.  
  \]  
\end{lemma}
\begin{proof}
  If $\curl \mu =0$, then obviously $\tau =0$. We claim that the
  converse is also true. Indeed, if $\tau=0$, then
  putting $s = d^{-1} \trace{\curl(\curl(\mu)B)}$, we have
  \begin{equation}
    \label{eq:curlcurl-t}
    \curl (\curl (\mu) B) = s  \,\id.    
  \end{equation}
  Taking divergence on both sides, we find that $\nabla s = 0$, so $s$
  must be a constant on $T$. Then, taking normal components of both
  sides of~\eqref{eq:curlcurl-t} on each
  facet, we find that $sn=0$, so $s=0$. Hence $\curl (\curl (\mu) B)=0$, which in turn
  implies that $0 = (\curl (\curl (\mu) B, \mu)_T = (\curl
  (\mu) B, \curl(\mu))_T=0$. Therefore, by \cite[Lemma~2.2]{MR2629995},
  $\curl(\mu)=0$. 

  Applying this on the reference element $\hat T$ for
  $\hat\mu = F^\trans (\mu \circ \phi) F \in \Poly^k(T, \mm) $ and 
  $\hat\tau = \Dev{\hatcurl(\hatcurl (\hat \mu) \hat B)}$ where
  $\hat{B}$ is in Remark~\ref{rem:implB}, by finite
  dimensionality, we have
  \begin{equation}
    \label{eq:1tau-mu}
    \| \hat{\tau} \|_{\hat T} \sim \| \hatcurl (\hat\mu) \|_{\hat T}.
  \end{equation}
  We will now show that $\tau = (\det F) \Dev{\curl(\curl ( \mu) B)}$ is
  related to $\hat \tau$ by 
  \begin{equation}
    \label{eq:Mtau}
    \tau  = \MM (\hat \tau).
  \end{equation}
  By the definition of $\MM$,
  \begin{align*}
    (\det F)\,\MM(\hat \tau) \circ \phi
    & = F^{-\trans}
      \Dev{
      \hatcurl
      (\hatcurl(\hat\mu) \hat{B})}
      F^\trans
    % \\
    % &
      =\Dev{
    F^{-\trans}
    \hatcurl(\hatcurl(\hat\mu) \hat{B})
    F^\trans}
  \end{align*}
  as  trace is preserved under similarity transformations. Focusing on
  the part of the last term inside the deviatoric, in the $d=3$ case, 
  \begin{align*}
    F^{-\trans}
    &     \hatcurl(\hatcurl(\hat\mu) \hat{B})
    F^\trans
      =
      F^{-\trans}\hatcurl\big[\hatcurl
      (  F^\trans{(\mu\circ \phi)}F) \, F^\trans {(B\circ\phi)} F\big]
      F^\trans
    &&\text{by~\eqref{eq:bubbletrafo},}
    \\
    &= F^{-\trans}\hatcurl
      \big[F^\trans [\curl(\mu) \circ \phi]\,
      F^{-\trans} (\det{F}) \,F^\trans {(B\circ\phi)} F\big] F^\trans
      && \text{by~\eqref{eq:curlFFt},}
    \\
    &= (\det{F}) F^{-\trans} \hatcurl
      \big[ F^\trans[\curl(\mu) B]  \circ\phi\, F \big]
      F^\trans \\
    &= (\det{F})^2 F^{-\trans}F^\trans
      \big[\curl  (\curl(\mu) B) \circ\phi
      \big]F^{-\trans} F^\trans
    && \text{by~\eqref{eq:curlFFt}.}% \\
    % &= (\det{F})^2 \curl(\curl(\mu) B) \circ\phi
  \end{align*}
  This proves that 
  \[
    F^{-\trans}
    \hatcurl(\hatcurl    (\hat\mu) \hat{B})
    F^\trans
    = (\det{F})^2 \curl(\curl(\mu) B) \circ\phi
  \]
  when $d=3$.
  The same identity holds in the $d=2$ case: the argument is similar 
  after changing the definitions of the curls and the mapping of $B$
  appropriately.
  Thus, 
  $\MM(\hat \tau) \circ \phi =
  (\det{F}) \Dev{\curl(\curl(\mu) B)} \circ\phi$
  and~\eqref{eq:Mtau} is proved.

  Finally, the result follows from~\eqref{eq:Mtau} by scaling
  arguments: indeed~\eqref{eq:1tau-mu} implies,
  by~\eqref{eq::sigmascaling} and~\eqref{eq:curlFFt} that
  \begin{alignat*}{3}
    h^3 \| \tau \|_T^2  & \sim h^3 \| \curl \mu \|_{T}^2
    && \quad\text{ if } d=3,
    \\
    h^2 \| \tau\|_T^2 & \sim h^4 \| \curl \mu\|_T^2
    && \quad\text{ if } d=2,
  \end{alignat*}
  from which the result follows.
\end{proof}

\begin{lemma} \label{lem::lbbone}
  For any  $\gamma_h \in W_h$, there is a $\tau_h \in \Sigmaplus$ such that
  \begin{align}
    \label{eq:tau-gamma}
    (\tau_h, \gamma_h)_\Omega
    \,\gtrsim \,h \| \curl \gamma_h \|_h\,
    \| \tau_h \|.
  \end{align}
  Furthermore, for any $v_h \in V_h,$ the same $\gamma_h, \tau_h$ pair
  satisfies
  \begin{align}
    \label{eq:b2-bdd-below}
    b_2(\tau_h, (v_h, \gamma_h)) \gtrsim
    \bigg[
    h \| \curl (\gamma_h) \|_h - \| \divergence(v_h) \|_h
    \bigg]  \| \tau_h \|.
  \end{align}
\end{lemma}
\begin{proof}
  Given a $\gamma_h \in W_h$, set $\tau_h$ element by element by
  \[
    \tau_h|_T = (\det F) \,\Dev{ \curl(\curl(\gamma_h|_T)B) }.
  \]
  Clearly,
  $\Dev{ \curl(\curl( \Pi^{k-1}\gamma_h)B) }$ is in $\Sigma_h.$
  Since $\Dev{ \curl(\curl( \gamma_h- \Pi^{k-1}\gamma_h)B) }$ is in
  $\delta\Sigma_h,$ we conclude that $\tau_h \in \Sigma_h^+$.
  Since $\gamma_h$ is trace-free,
  $ (\tau_h, \gamma_h)_T = (\curl(\curl(\gamma_h|_T)B), \gamma_h)_T$
  \,$\det F,$ which in turn implies, after integrating by parts and
  applying Lemma \ref{lem::bubbletransform},
  $  (\tau_h, \gamma_h)_T = (\curl(\gamma_h){B}, \curl
  \gamma_h)_T$\,$\det F$.

  In the $d=3$ case, this yields
  \begin{equation}
    \label{eq:sumB}
    (\tau_h, \gamma_h)_T
    = \det F \, \int_T \sum_{i=0}^3 \lambda_{i-3} \lambda_{i-2}\lambda_{i-1}
      |\curl(\gamma_h) \nabla \lambda_i|^2\dx
  \end{equation}
  Noting that  $\grad \lambda_i = -n_i/h_i$, where $h_i$ is the distance
  from the $i$th vertex to the facet of the simplex opposite to it,
  and that the $\ell^2$-norm of any matrix $m \in \mm$ is equivalent to
  the sum of $\ell^2$-norms of $m n_i$,
  a local scaling argument with $m = \curl(\gamma_h)$
  and~\eqref{eq:sumB} imply
  \[
    (\tau_h, \gamma_h)_T \gtrsim (\det F) h^{-2} \| \curl (\gamma_h)\|_T^2.
  \]
  Therefore, 
  $    (\tau_h, \gamma_h)_\om \gtrsim h \| \curl (\gamma_h)\|_h^2
  \gtrsim h \| \curl (\gamma_h)\|_h\,\| \tau_h\|$,
  by Lemma~\ref{lem:bubble-equiv}. This proves~\eqref{eq:tau-gamma} in
  the $d=3$ case.
  In the $d=2$ case, the analogue of~\eqref{eq:sumB} gives
  $
    (\tau_h, \gamma_h)_T \gtrsim $ $(\det F)$  $ \| \curl
    (\gamma_h)\|_T^2$
    $\gtrsim h^2    \| \curl    (\gamma_h)\|_T^2
    \ge h    \| \curl    (\gamma_h)\|_T \,\| \tau_h\|,
  $
  where we have used Lemma~\ref{lem:bubble-equiv} again. This
  completes the proof  of~\eqref{eq:tau-gamma}.

  To prove~\eqref{eq:b2-bdd-below},  we use~\eqref{eq::blftwoequitwo}.
  The last sum in
  \begin{align*}
    \blftwo(\tau_h, ( v_h,\gamma_h))
    &=
      -\sum\limits_{T \in \mesh} \int_T \tau_h : (\nabla v_h - \gamma_h) \dx + \sum\limits_{F \in \facets} \int_F (\tau_h)_{nt} \cdot \jump{(v_h)_t} \ds
  \end{align*}
  vanishes due to Lemma~\ref{lem::deltaspacenormtang}. Hence
  by~\eqref{eq:tau-gamma},
  \begin{align}
    \label{eq:2}
    \blftwo(\tau_h, ( v_h,\gamma_h))
    & \,\gtrsim \,h \| \curl \gamma_h\|_h\, \| \tau_h \|
      -\sum\limits_{T \in \mesh}
      (\tau_h, \nabla v_h)_T.
  \end{align}
  To handle the last term, note that 
  \begin{align*}
    \frac{1}{\det F} (\tau_h, \nabla v_h)_T
    & =    (\curl(\curl (\gamma_h)B), \nabla v_h)_T
      -    (d^{-1} \trace {\curl(\curl (\gamma_h)B)} \id, \nabla v_h)_T
    \\
    & = -  (d^{-1}  \trace {\curl(\curl (\gamma_h)B)}, \divergence( v_h))_T
    % \\
    % & \gtrsim -h^{-1} \| B \|_\infty \| \curl (\gamma_h) \|_T \| \divergence(v_h)\|_T
  \end{align*}
  because
  $(\curl(\curl (\gamma_h)B), \nabla v_h)_T=0$. This follows by  integrating
  one of the curls by parts, observing that 
  the resulting volume term
  is zero (since $\curl (\nabla v_h) =0$) and so is the resulting boundary term 
  (due to  Lemma~\ref{lem::bubbletransform}).
  Continuing, we apply Cauchy-Schwarz inequality and an inverse
  inequality to get 
  \begin{align*}
    |(\tau_h, \nabla v_h)_T|
    & \lesssim
    |\det F| h^{-1} \| B \|_{L^\infty(T)}\|\curl (\gamma_h)\|_T
      \|\divergence( v_h)\|_T
    \\
    & \lesssim \|\tau_h\|_T      \|\divergence( v_h)\|_T
  \end{align*}
  by Lemma~\ref{lem:bubble-equiv}. 
  Returning to~\eqref{eq:2} and using this estimate, the proof is complete.
\end{proof}

\begin{remark}
  The message of Lemmas~\ref{lem:bubble-equiv} and~\ref{lem::lbbone}
  is that it is possible to choose a $\tau_h$ in the form of a
  deviatoric of a curl of a bubble to bound (from below) the term
  arising from the weak symmetry constraint. If $\tau_h$ was just a
  curl, it would not be seen by the equilibrium equation and the bound
  in~\eqref{eq:b2-bdd-below} would not have the
  $\|\divergence(v_h)\|$-term, but our $\tau_h$ is a 
  deviatoric (of a curl), thus necessitating this term.  
\end{remark}

\begin{lemma} \label{lem::lbbtwo}
  For any  $(v_h, \gamma_h) \in U_h,$ there is a $\tau_h
  \in \Sigma_h$ such that 
  \begin{align*}
    b_2(\tau_h, (v_h, \gamma_h))\gtrsim \| (v_h, \gamma_h) \|_{U_h,*}
    \| \tau_h \|.
  \end{align*}
  % and $\| \tau_h \|_{\Sigmaplus} \lesssim  \| (v_h, \gamma_h) \|_{U_h,*}$.
\end{lemma}
\begin{proof}
  We only present the proof in two dimensions, as the three
  dimensional case is similar.  From the local element basis exhibited
  in~\eqref{eq:Sbasis} (see also \cite[\S5.5]{gopledschoe} for a more
  detailed discussion), its clear that on any facet $F \in \facets$,
  there exists a constant trace-free function $S^F$ with the property
  that $S^F_{nt} \in \Poly^0(F, n_F^\perp)$, $\|S^F_{nt}\|_2 = 1$ on
  the facet $F,$ and $S^F_{nt}$ equals $(0,0)$ on all other facets in
  $\mathcal{F}_h$.  Given any $(v_h, \gamma_h) \in U_h$, define
  \begin{align*}
    \tau_h^0 &:= \sum\limits_{T \in \mesh} \sum\limits_{F \in \facets}
               \!\!-(S^F : \Pi^{k-1}\Dev{\nabla v_h - \gamma_h}) \,
               \lambda_T^F\,S^F, &\quad
    \tau_h^1 &:= \sum\limits_{F \in \facets}  \frac{1}{\sqrt{h}}
               \Pi^{1} (\jump{(v_h)_t})\; S^F, 
  \end{align*}
  where $\lambda_T^F$ is the unique barycentric coordinate function on
  the element $T$ opposite to the facet $F$ (so that
  $\lambda_T^FS^F$ is an $nt$-bubble).
  Clearly, $\tau_h^0$ and $\tau_h^1$ are in $\Sigma_h$. 
  Using the norm equivalences stated in \eqref{eq::sigmanormequi} and the
  mappings for $v_h$ and $\gamma_h$ given
  in~\eqref{eq:map_PW},  a scaling argument yields
  \begin{align*}
    \|\tau_h^0\|^2 \lesssim \sum\limits_{T \in \mesh} \| \Pi^{k-1}\Dev{\nabla v_h - \gamma_h}) \|_T^2  \quad \textrm{and} \quad \|\tau_h^1\|^2 \lesssim \sum\limits_{F \in \facets}\frac{1}{h} \| \Pi^{1} \jump{(v_h)_t} \|_F^2.
  \end{align*}
  Setting $\tau_h = \alpha_0 \tau_h^0 + \alpha_1\tau_h^1$ and
  selecting the constants $\alpha_0, \alpha_1$ appropriately, the rest
  of the proof proceeds along the same lines as the proof of
  \cite[Lemma 6.5]{gopledschoe}.
\end{proof}

\begin{remark}
  It is interesting to contrast Lemma~\ref{lem::lbbtwo} with
  \cite[Lemma~6.5]{gopledschoe}. The latter gives a similar
  LBB-condition. The differences are (i) the velocity space in
  \cite{gopledschoe} is $\BDM^{k+1}$ (defined in
  Remark~\ref{rem:BDM1}), (ii) the velocity norm is a discrete
  $H^1$-norm defined using $\nabla$ in place of $\eps(\cdot)$, (iii) there is
  no weak symmetry constraint and no associated space $W_h$, and (iv)
  the stress space in \cite{gopledschoe} equals the $\Sigma_h$ in
  \eqref{eq:Sigma_h} plus certain $nt$-bubbles of degree $k+1$
  (different from our $\delta\Sigma_h$ here). Lemma~\ref{lem::lbbtwo}
  shows that the inf-sup condition in \cite[Lemma~6.5]{gopledschoe}
  continues to hold even {\em if the $nt$-bubbles there are removed}
  and $\BDM^{k+1}$ is replaced by our Raviart-Thomas velocity space
  $V_h$. This observation can be extended to prove  the
  convergence of the MCS formulation in~\cite{gopledschoe}
  with so modified spaces.
%Further, it is easy to see that $\divergence(\sigma_h + q_h) \in [\RT^k]^*$ for all $(\sigma, q_h) \in \Sigma_h \times Q_h$. A simple counting argument (and the surjectivity given by the LBB-condition) then shows $\divergence(\Sigma_h \times Q_h) = [\RT^k]^*$.
\end{remark}

\begin{theorem}[Discrete LBB-condition] \label{theorem::LBB}
  Let $v_h \in V_h$ and $\gamma_h \in W_h$.
  Then,
  \begin{align}
        \label{eq:lbb-1}
    \sup\limits_{(\tau_h, q_h) \in \Sigmaplus \times Q_h} \frac{b_1( v_h, q_h) + b_2(\tau_h, (v_h, \gamma_h))}{\| \tau_h \| + \| q_h \|} \gtrsim \| (v_h, \gamma_h) \|_{U_h}.
  \end{align}
  If $v_h$ is in the divergence-free subspace
  $V_h^0 := \{ z_h \in V_h: \divergence(z_h) = 0\},$ then 
  \begin{align}
    \label{eq:lbb-b2only-2}
    \sup\limits_{\tau_h \in \Sigmaplus}
    \frac{ b_2(\tau_h, (v_h, \gamma_h))}{\| \tau_h \|}
    \gtrsim \| (v_h, \gamma_h) \|_{U_h}.
  \end{align}
\end{theorem}
\begin{proof}
  By Lemmas~\ref{lem::lbbone} and~\ref{lem::lbbtwo}, for any given
  $(v_h, \gamma_h) \in U_h$, there are
  $\tau_h^1, \tau_h^2 \in \Sigma_h^+$ satisfying
  \begin{align}
    \label{eq:4}
    b_2(\tau_h^1, (v_h, \gamma_h))
    & \gtrsim
      \bigg[
      h \| \curl (\gamma_h) \|_h - \| \divergence(v_h) \|
      \bigg]  \| \tau^1_h \|,
    \\ \label{eq:5}
    b_2(\tau^2_h, (v_h, \gamma_h))
    & \gtrsim \| (v_h, \gamma_h) \|_{U_h,*}
      \| \tau^2_h \|,
  \end{align}
  Clearly, the same inequalities hold when $\tau^1_h$ and $\tau^2_h$
  are scaled by any nonzero factor, so we may assume without loss of
  generality, that they have been scaled so that
  $ \|\tau_h^1 \| = h \| \curl \gamma_h \|_h $
  and $ \|\tau^2_h \| = \| (v_h, \gamma_h) \|_{U_h,*}.$
  Set $\tau_h = \alpha \tau_h^1 + \tau^2_h,$ where
  $\alpha \in \rr$ is to be chosen shortly. It follows
  from~\eqref{eq:4} and~\eqref{eq:5} that
  \begin{equation}
    \label{eq:3_b2}
    b_2(\tau_h, (v_h, \gamma_h)) \gtrsim
    \alpha   
    h^2 \| \curl \gamma_h \|^2_h -
    \alpha   h \| \divergence(v_h) \|_h \| \curl \gamma_h\|_h
    +
    \| (v_h, \gamma_h) \|_{U_h,*}^2.
  \end{equation}

  Next, we choose $q_h \in Q_h$ so that
  $q_h =\beta \divergence(v_h)$, where $\beta \in \rr$ is another
  constant to be chosen shortly.  Then~\eqref{eq:3_b2} implies
  \begin{align*}
    b_1( v_h, q_h) + b_2(\tau_h, (v_h, \gamma_h))
    & =
      \beta \|\divergence(v_h)\|_{h}^2
      +
      \alpha   
      h^2 \| \curl \gamma_h \|^2_h
      +
      \| (v_h, \gamma_h) \|_{U_h,*}^2
    \\
    &  -
      \alpha   h \| \divergence(v_h) \|_h \| \curl \gamma_h\|_h.
  \end{align*}
  Choose any $\alpha >1$ and $\beta > \alpha^2/2$.
  Then,  using Young's inequality for the last term,
  \begin{align*}
    b_1( v_h, q_h) + b_2(\tau_h, (v_h, \gamma_h))
    & \gtrsim
      \|\divergence(v_h)\|_{h}^2  +  h^2 \| \curl \gamma_h \|_h^2  +
      \|(v_h, \gamma_h) \|_{U_h,*}^2.
    \\
    \intertext{Recalling that we also have}
    \| \tau_h \|_{\Sigma_h^+}^2 + \| q_h \|^2
    & \lesssim   \|\divergence(v_h)\|_h^2 +  h^2 \| \curl \gamma_h\|_h^2
      + \| (v_h, \gamma_h) \|_{U_h,*}^2,
  \end{align*}    
  we can now conclude the proof of~\eqref{eq:lbb-1}
  using the norm equivalence of Lemma~\ref{lem::normequi}. The proof
  of~\eqref{eq:lbb-b2only-2} is similar (and in fact simpler since all
  terms involving $\divergence(v_h)$ vanish).
\end{proof}

\subsection{Error estimates}

In this subsection we show that the error in the discrete MCS solution
converges at optimal order. As we have chosen polynomials of degree $k$
for the stress space $\Sigma_h$, the optimal rate of convergence for
$\| \sigma-\sigma_h\|$ is $\mathcal{O}(h^{k+1})$.
However, the optimal rate for the velocity error in our discrete
$H^1$-like norm, namely, $\|u - u_h\|_{1,h,\eps}$ is only
$\mathcal{O}(h^{k})$ (since the Raviart-Thomas velocity space $V_h$
only contains $\Poly^k(T, \rr^d)$ within each mesh element $T$).
Nevertheless, we are still able to prove optimal convergence rate of
the stress error by using an appropriate interpolation operator and
deducing that the stress error is independent of the velocity error.
Another important property we shall conclude in this subsection
is the pressure-robustness of the method.
% This property was first discussed in the work \cite{linke2014role}. The idea is, that in the infinite dimensional setting, the velocity is only steered by the rotational part of a continuous Helmholtz decomposition of a given load $f$. If we can inherit this property also in the discrete setting, the method is called pressure robust as it allows (using a standard velocity-pressure formulation) $H^1$-like velocity errors that are independent of the pressure. This property further yields a robustness with respect to the viscosity $\nu$. In \cite{gopledschoe} we have proven that the standard MCS method is also pressure robust. In this work we present a similar result proving that the stress error $\| \sigma-\sigma_h\|$ is not only independent of the velocity error (as mentioned above) but also independent of the pressure.

\begin{lemma}[Continuity] \label{lem::continuity}
  The bilinear forms $a, b_1$ and $b_2$ are continuous: 
  \begin{alignat*}{2}
    a(\varsigma_h, \tau_h) &\lesssim
    \big( \nu^{-1/2}
    \|\varsigma_h\|\big)
    \big( \nu^{-1/2} \|\tau_h\|\big),  &&
    \quad \textrm{ for all } \varsigma_h, \tau_h \in \Sigmaplus,
    \\
    b_1(v_h, q_h) &\lesssim  \|(v_h, 0)\|_{U_h}\, \|q_h\|,  && \quad \textrm{ for all } v_h\in V_h, q_h \in Q_h,\\
    b_2(\tau_h, (v_h,\eta_h)) &\lesssim \|\tau_h\|\; \|(v_h, \eta_h)\|_{U_h},   && \quad \textrm{ for all } \tau_h \in \Sigmaplus, (v_h, \eta_h) \in U_h.
  \end{alignat*}
\end{lemma}
\begin{proof}
  The continuity of $a$ and $b_1$ follow by the Cauchy Schwarz
  inequality. For $b_2,$ we use \eqref{eq::blftwoequitwo}
  and $\nabla v_h = \eps(v_h) +
  \kappa(\curl v_h)$ to get 
  \begin{align*}
  \blftwo(\tau_h, (v_h,\eta_h))
    &=  -\sum\limits_{T \in \mesh} \int_T \tau :
      \big[\eps(v_h)  + 
      (\kappa(\curl v_h) - \eta_h) \big]\dx
      + \sum\limits_{F \in \facets} \int_F \tau_{nt} \cdot \jump{(v_h)_t} \ds.
  \end{align*}
  Now, Cauchy-Schwarz inequality and \eqref{eq::sigmanormequi}
  of Lemma~\ref{lem:equiv1} finishes the proof.
\end{proof}

\begin{lemma}[Coercivity in the kernel] \label{lem::coerzive}
  For all $(\tau_h, q_h)$ in the kernel 
  \begin{align*}
    \kernel := \{ (\tau_h,q_h) \in  \Stressspaceh
    \times \Presspaceh: \blfone(\Velvarhtest, \Presvarhtest) +
    \blftwo(\tau_h,(\Velvarhtest, \eta_h))= 0 \textrm{ for all }  (\Velvarhtest, \eta_h) \in U_h \},
  \end{align*}
  we have
  $\nu^{-1}
  \big(\,\Sigmanormh{ \tau_h } +
    \| q_h \|
    \big)^2  \lesssim \;
    \ablf(\tau_h, \tau_h).
    $
\end{lemma}
\begin{proof}
  By~\cite[Theorem~2.2]{LedererSchoeberl2017}, for any $q_h\in Q_h,$
  there is a $v_h \in V_h$ such that $\| q_h \|^2
  \lesssim (\divergence(\Velvarhtest), q_h)$
  and a discrete $H^1$-norm of $v_h$ is bounded by $\| q_h
  \|$. The latter bound implies, in particular, that
  $ \| v_h\|_{1,h,\eps} \lesssim \| q_h \|,$ and also that
  $\eta_h=\kappa(\curl v_h)$ satisfies
  $ \| (v_h, \eta_h) \|_{U_h} \lesssim \| q_h \|.$ This together
  with Lemma~\ref{lem::continuity} implies
  \begin{align*}
    \| q_h \|^2
    & \lesssim b_1(v_h, q_h) = -b_2( \tau_h, (v_h, \eta_h))
      \lesssim \,\|\tau_h\| \,\|(v_h, \eta_h)\|_{U_h}
      \lesssim \,\|\tau_h\| \,\| q_h\|
  \end{align*}
  yielding the needed bound for $\| q_h\|$. 
\end{proof}

We are now ready to conclude an inf-sup condition for 
$
 \blfbig(v_h, \eta_h,
\tau_h, q_h;
\tilde{v}_h,\tilde{\eta}_h, \tilde{\tau}_h, \tilde{q}_h)
:=\ablf(\tau_h, \tilde{\tau}_h) + 
\blfone(v_h, \tilde{q}_h) +
\blfone(\tilde{v}_h, q_h) +
\blftwo(\tau_h,(\tilde{v}_h,\tilde{\eta}_h)) +
\blftwo(\tilde{\tau}_h,(v_h, \eta_h)). 
$
\begin{corollary} \label{cor:big_inf_sups}
  Let $\tau_h \in \Sigma_h^{+}$, $v_h \in V_h$,
  $\eta_h \in W_h$, and $ q_h\in Q_h$. There holds
  \begin{align}
    \label{eq:big-inf-sup-1}
    \|(v_h, \eta_h, \tau_h, q_h) \|_*
    \lesssim
    \sup
    \limits_{\substack{\tilde{v}_h \in V_h, \; \tilde{\eta}_h \in W_h \\
    \tilde{\tau}_h \in \Sigmaplus,\;\tilde{q}_h \in Q_h    }}
    \frac{\blfbig(v_h, \eta_h, \tau_h,q_h;
    \tilde{v}_h, \tilde{\eta}_h, \tilde{\tau}_h, \tilde{q}_h)}
    {\|(\tilde{v}_h, \tilde{\eta}_h, \tilde{\tau}_h, \tilde{q}_h) \|_*},
  \end{align}
  so, in particular, there is a unique solution for the discrete MCS
  system~\eqref{eq::discrmixedstressstokesweak}. Moreover, if $v_h$ is
  restricted to 
  $V_h^0$, we also have
  \begin{align}
        \label{eq:big-inf-sup-2}
    \|(v_h, \eta_h, \tau_h, 0) \|_* \lesssim
    \sup\limits_{\tilde{v}_h \in V^0_h,
    \;\tilde{\eta}_h \in  W_h,\;
    \tilde{\tau}_h \in\Sigmaplus
    }
    \frac{
    \blfbig(v_h, \eta_h, \tau_h, 0;
    \tilde{v}_h, \tilde{\eta}_h,  \tilde{\tau}_h, 0)}
    {\|(\tilde{v}_h, \tilde{\eta}_h,  \tilde{\tau}_h, 0) \|_*}.
  \end{align}
\end{corollary}
\begin{proof}
  The first inf-sup condition follows from the standard theory of
  mixed methods~\cite{brezzi2012mixed}, using
  Theorem~\ref{theorem::LBB} (the inf-sup condition for $b_1$ and
  $b_2$ given by~\eqref{eq:lbb-1}), Lemma~\ref{lem::continuity}
  (continuity of forms), and Lemma~\ref{lem::coerzive} (coercivity in
  the kernel).

  The second inf-sup condition also follows in a similar fashion, but
  now using the other inequality \eqref{eq:lbb-b2only-2} of
  Theorem~\ref{theorem::LBB}.
\end{proof}

\begin{theorem}[Consistency] \label{the::consistency}
  The MCS method with weakly imposed
  symmetry~\eqref{eq::discrmixedstressstokesweak} is consistent in the
  following sense. If the exact solution of the Stokes
  problem~\eqref{eq::mixedstressstokes} is such that
  $\Velvar \in H^1(\om, \rr^d)$, $\omega \in L^2(\om, {\mm} )$,
  $\Stressvar \in H^1(\om, \mathbb{D})$ and
  $\Presvar \in L^2_0(\om, \rr)$, then
  \begin{align*}
      \blfbig(\Velvar,\omega, \Stressvar, \Presvar;\Velvarhtest,\eta_h, \Stressvarhtest, \Presvarhtest) 
                                   %\ablf(\Stressvar, \Stressvarhtest) + \blftwo(\Stressvarhtest, (\Velvar, \omega) ) + \blftwo(\Stressvar, (\Velvarhtest, \eta_h)) + \blfone(\Velvarhtest, \Presvar) + \blfone(\Velvar, \Presvarhtest)
                                   = (-\Forcevar, \Velvarhtest)_\om
  \end{align*}
  for all $\Velvarhtest \in \Velspaceh, \eta_h \in W_h, \Presvarhtest \in \Presspaceh,
  $ and $\Stressvarhtest \in \Stressspaceh.$
%   Furthermore, we have consistency on the subspace of divergence free
%   velocity fields in the sense that 
%   \begin{align*}
%     \blfbig(\Velvar,\omega, \Stressvar, 0;\Velvarhtest,\eta_h, \Stressvarhtest, 0) 
% %   \ablf(\Stressvar, \Stressvarhtest) +
% %  \blftwo(\Stressvarhtest, (u, \omega) ) + 
% %  \blftwo(\Stressvar, (\Velvarhtest, \eta_h))
%     = (-\Forcevar, \Velvarhtest)_\om
%   \end{align*}
%   for all $\Velvarhtest \in V_h^0, \eta_h \in W_h$, and $\Stressvarhtest \in \Stressspaceh.$ 
\end{theorem}
% \begin{proof}
%   The proof follows exactly the same lines (now including the skew-symmetric function $\omega$) as the proof of Theorem 6.2 in \cite{gopledschoe}. For the second result note that
% \begin{align*}
%   %\blfbig(\Velvar,\Stressvar, 0;\Velvarhtest,\Stressvarhtest, 0)
%       \blfbig(\Velvar,\omega, \Stressvar, 0;\Velvarhtest,\eta_h, \Stressvarhtest, 0)= \int_\om \divergence(\Stressvar) \cdot \Velvarhtest = \int_\om -f \cdot \Velvarhtest + \int_\om \nabla \Presvar \cdot \Velvarhtest = \int_\om -f \cdot \Velvarhtest 
% \end{align*}
% for all $v_h \in V_h^0$. The rest follows similarly.
% \end{proof}

The proof of Theorem~\ref{the::consistency} is easy (see, e.g., the
similar proof of \cite[Theorem 6.2]{gopledschoe}), so we omit it.  We
now have all the ingredients to prove the following convergence
result.  Let $I_{V_h}$ denote the standard Raviart-Thomas interpolator
(see, e.g.,~\cite{brezzi2012mixed}) and let
$\|(u, \omega, \sigma, p)\|_{\nu, s} = \nu^{-1} \| \sigma\|_{H^s(\mesh,
  \dd)} + \nu^{-1} \|p \|_{H^s(\mesh, \rr)} + \| \omega \|_{H^s(\mesh,
  \kk)} + \| u \|_{H^{s+1}(\mesh, \rr^d)}.$

\begin{theorem}[Optimal convergence] \label{the::errorconvergence}
  Let $\Velvar \in H^1(\om, \rr^d) \cap H^m(\mesh, \rr^d)$,
  $\Stressvar \in  H^1(\om, \mathbb{D}) \cap H^{m-1}(\mesh,
  \mathbb{D})$, $\Presvar \in  L^2_0(\om, \rr) \cap H^{m-1}(\mesh,
  \rr)$ and $\omega \in  L^2(\om, \kk) \cap H^{m-1}(\mesh, \kk)$ be
  the exact solution of the mixed Stokes problem
  \eqref{eq::mixedstressstokes}, let $u_h$,$\Stressvarh$, 
  $\omega_h$ and $p_h$ solve~\eqref{eq::discrmixedstressstokesweak}
  and let $s = \min(m-1,k+1)$. Then, 
  \begin{equation}
    \label{eq:1}
    \begin{aligned}
      \frac{1}{\nu}(\Sigmanormh{ \Stressvar - \Stressvarh} 
      + \| p -      p_h\|)
      &+ \| (\omega_h - \Pi^k\omega, u_h - I_{V_h}u)\|_{U_h}
      \lesssim h^s  \| (0,  \omega, \sigma, p) \|_{\nu, s}.
      % \lesssim h^s (\frac{1}{\nu} \|\sigma\|_{H^{s}(\mesh)} + \frac{1}{\nu} \|p\|_{H^{s}(\mesh)} + \|\omega\|_{H^{s}(\mesh)}  ).
      % + \frac{1}{\visc} \Presnormh{\Presvar - \Presvarh}
      % + \frac{1}{\visc} \|\Stressvar\|_{H^{s}(\mesh)
    \end{aligned}    
  \end{equation}
\end{theorem}
\begin{proof}
  Let $e_h^\sigma = I_{\Sigma_h}\sigma - \sigma_h$,
  $e_h^u = I_{V_h}u - u_h$,
  $e_h^\omega = \Pi^k\omega - \omega_h$, $e_h^p = \Pi^kp - p_h$ (where
  the two occurrences of $\Pi^k$ represent projections onto two
  different discrete spaces per our prior notation).
  Denoting the analogous approximation errors by 
  $a^\sigma = I_{\Sigma_h}\sigma - \sigma$,
  $a^u = I_{V_h} u - u$,
  $a^\omega = \Pi^k\omega - \omega$, and 
  $a^p = \Pi^kp - p$, observe that 
  Theorem~\ref{the::consistency} implies 
  \begin{align}
    \label{eq:6}
    \blfbig(e_h^u,
    & e_h^\omega, e_h^\sigma, e_h^p;\;
      \Velvarhtest,\eta_h,\Stressvarhtest,q_h)
      =
    \blfbig(a^u, a^\omega, a^\sigma, a^p;\;
      \Velvarhtest,\eta_h,\Stressvarhtest,q_h)
  \end{align}
  for any $v_h \in V_h, \eta_h \in W_h, \tau_h \in \Sigmaplus,$ and
  $ q_h \in Q_h$. The right hand side above is a sum of five terms
  $(\nu^{-1}a^\sigma, \tau_h) + b_1(a^u, q_h) + b_1(v_h, a^p) +
  b_2(\tau_h, (a^u, a^\omega)) + b_2(a^\sigma, (v_h, \eta_h)).$ The
  second term vanishes:
  $ b_1(a^u, q_h) = (\div(I_{V_h} u - u), q_h) = (\Pi^{k} \div (u) -
  \div (u), q_h)=0$ as $\div(u) =0$. The third term also vanishes:
  $b_1(v_h, a^p) = (\div(v_h), \Pi^k p - p)=0$ since
  $\div(v_h) \in \Poly^k(\mesh)$. The fourth term, due
  to~\eqref{eq::blftwoequione}, is
  \begin{align*}
    b_2(\tau_h, (a^u, a^\omega))
    &= (\tau, a^\omega) + \sum\limits_{T \in \mesh}
      (\divergence(\tau_h),  I_{V_h}u - u)_T
      % \int_T\!\! \divergence(\tau_h) \cdot (I_{V_h}u - u) \dx
      % - \!\!\sum_{F \in \facets} \!\!\int_{F} \jump{(\tau_h)_{nn}}(I_{V_h}u - u)_n \!\ds
      - \!\sum_{E \in \facets}
      (\jump{(\tau_h)_{nn}}, (I_{V_h}u - u)\cdot n)_E
    % \\
    %   % &+\sum\limits_{T \in \mesh} \int_T \tau_h :(\Pi^k\omega - \omega) \dx \\
    %                                                       &\lesssim \| \tau_h\| \| (\Pi^k\omega - \omega) \|_{L^2} \lesssim \frac{1}{\sqrt{\nu}}\| \tau_h\| \sqrt{\nu} \| (0,\Pi^k\omega - \omega) \|_{U_h}.
  \end{align*}
  where the last two terms vanish by the properties of the
  Raviart-Thomas d.o.f.s that define $I_{V_h}$, i.e., 
  % (since $\div(\tau_h)|_T \in \Poly^{k-1}(T,
  % \rr^d)$ and $(\tau_h)_{nn}|_F \in \Poly^{k}(F)$).
  $b_2(\tau_h, (a^u, a^\omega))
    =  (\tau_h, a^\omega).$
  The fifth term,
  due to~\eqref{eq::blftwoequitwo}, is 
  \begin{align*}
    b_2(a^\sigma, (v_h, \eta_h))
    & =
      (a^\sigma, \eta_h - \grad v_h) 
      % + \sum\limits_{F \in \facets} \int_{F} a^\sigma_{nt} \cdot
      % \jump{(v_h)_t} \ds.
      + \sum\limits_{E \in \facets}
      (a^\sigma_{nt}, \jump{(v_h)_t})_E
  \end{align*}
  Writing
  $(a^\sigma, \eta_h - \grad v_h) = (a^\sigma, \eta_h) + (a^\sigma,
  (\Pi^{k-1} - \id) \grad v_h ) -(a^\sigma, \Pi^{k-1} \grad v_h),$
  note that by the d.o.f.s of Theorem~\ref{theo::unisolvent}, the last term
  $(a^\sigma, \Pi^{k-1} \grad v_h)$  is
  zero, and moreover,
  $ (a^\sigma, \eta_h) = (a^\sigma, \eta_h - \Pi^0\eta_h)$.
  Incorporating these observations on each term
  into~\eqref{eq:6}, we obtain
  \begin{equation}
    \label{eq:3}
    \begin{aligned}
      \blfbig(e_h^u, e_h^\omega, e_h^\sigma, e_h^p;
      \;
      \Velvarhtest,\eta_h,\Stressvarhtest,q_h)
      & = (\nu^{-1}a^\sigma, \tau_h) + (\tau_h, a^\omega) 
        + \sum\limits_{F \in \facets} (a^\sigma_{nt}, \jump{(v_h)_t})_F
        \\ 
      & 
        + (a^\sigma, \eta_h - \Pi^0\eta_h) 
        + (a^\sigma,  (\Pi^{k-1} - \id) \grad v_h)
    \end{aligned}    
  \end{equation}
  
  We now proceed to estimate the right hand side of~\eqref{eq:3}.
  By~\eqref{eq:singlenorms-1} and Lemma~\ref{lem:gradofRT},
  \begin{gather*}
    \|\eta_h - \Pi^0\eta_h\|
     \lesssim h \| \nabla \eta_h \|_h
      \lesssim \inf\limits_{\tilde{v}_h \in V_h} \| (\tilde{v}_h,
      \eta_h) \|_{U_h}
      \le \| (v_h, \eta_h) \|_{U_h},
    \\
    \|(\Pi^{k-1} - \id) \grad v_h \|_h
     \lesssim \| \div(v_h)\|^2
      \lesssim \| \eps(v_h)\|_h^2
      \le \| (v_h, \eta_h) \|_{U_h}.
  \end{gather*}
  Using these after an application of the Cauchy-Schwarz inequality,
  \eqref{eq:3} yields
  \begin{align}
    \nonumber
    \blfbig(e_h^u,
    & e_h^\omega, e_h^\sigma, e_h^p;\;
      \Velvarhtest,\eta_h,\Stressvarhtest,q_h)
    \\ \nonumber
    % & \lesssim
    %   \left(
    %   \frac 1 \nu
    %   \| a^\sigma\|^2
    %   + \nu \| a^\omega\|^2
    %   \right)^{1/2}
    %   \left(
    %   \frac 1 \nu
    %   \| \tau_h\|^2
    %   + \nu \| \eta_h - \Pi^0 \eta_h\|^2
    %   + \nu \| (\Pi^{k} - \id) \grad v_h \|_h^2
    %   \right)^{1/2}
    % \\
    & \lesssim
      \bigg[
      \frac 1 \nu
      \bigg(
      \| a^\sigma\|^2
      +
      \sum_{F \in \facets} h \|  a^\sigma_{nt} \|_F^2
      \bigg)
      +
      \nu \| a^\omega\|^2
      \bigg]^{1/2}
      \left(
      \frac 1 \nu
      \| \tau_h\|^2
      + \nu \| (v_h, \eta_h)  \|_{U_h}^2
      \right)^{1/2}
    \\ \label{eq:8}
    & \lesssim
      \bigg(\frac{1}{\sqrt{\nu}}
      h^s \| \sigma\|_{H^s(\mesh)}
      +
      \sqrt{\nu}
      h^s \| \omega\|_{H^s(\mesh)}
      \bigg)
      \| (v_h, \eta_h, \tau_h, q_h) \|_*,
  \end{align}
  where we have used Theorem~\ref{the::sigmainterpolation} and the 
  approximation property of $\Pi^k$ in the last step.

  To complete the proof, we apply triangle inequality starting from
  the left hand side of~\eqref{eq:1}, to get 
  \begin{align}
    \nonumber 
    \frac{1}{\nu}
     \| \Stressvar - \Stressvarh \|
    & + \frac{1}{\nu} \| p - p_h \|
    + \| (e_h^u, e_h^\omega) \|_{U_h}
    \le
      \frac{1}{\nu}
      \big(\| a^\sigma \|  + \| a^p\|
      +  \| e_h^\sigma \| + \| e_h^p\|
      \big)
    + \|  (e_h^u, e_h^\omega)\|_{U_h}
    \\ \label{eq:7}
    & \lesssim 
      \frac{h^s}{\nu}
      \big(\| \sigma\|_{H^s(\mesh)} + \| p \|_{H^s(\mesh)}
      \big)
      +
      \frac{1}{\sqrt{\nu}}
      \|(e_h^u,  e_h^\omega, e_h^\sigma, e_h^p) \|_*
  \end{align}
  again using~Theorem~\ref{the::sigmainterpolation}. Bounding the last term above
  using \eqref{eq:big-inf-sup-1} and ~\eqref{eq:8},
  \begin{align*}
    \frac{1}{\sqrt{\nu}}
    \|(e_h^u,  e_h^\omega, e_h^\sigma, e_h^p) \|_*
    & \lesssim    \sup
      \limits_{\substack{\tilde{v}_h \in V_h, \; \tilde{\eta}_h \in W_h \\
    \tilde{\tau}_h \in \Sigmaplus,\;\tilde{q}_h \in Q_h    }}
    \frac{ \blfbig(e_h^u, e_h^\omega, e_h^\sigma, e_h^p;
    \Velvarhtest,\eta_h,\Stressvarhtest,q_h)}
    { \sqrt{\nu} \|(v_h,\eta_h,\tau_h, q_h) \|_*}
    \lesssim h^s \| (0, \omega, \sigma, p)\|_{\nu, s},
  \end{align*}
  the proof is complete.
\end{proof}

\begin{remark}[Convergence in standard norms]
  \label{rem:superconvergence}
  Using also Lemma~\ref{lem::singlenorms}'s
  estimate~\eqref{eq:singlenorms-2}, a consequence of
  the global discrete Korn inequality,~\eqref{eq:1} implies 
  \begin{align}
    \label{eq:9}
    \frac{1}{\nu}\Sigmanormh{ \Stressvar - \Stressvarh}
    +
    \frac{1}{\nu}\|p - p_h \|
    +
    \| \omega - \omega_h \| + \| u_h - I_{V_h}u \|_{V_h}
    \lesssim
    h^{k+1} \| (0, \omega, \sigma, p) \|_{\nu, s}
  \end{align}
  under the assumptions of Theorem~\ref{the::errorconvergence} for a
  sufficiently smooth solution.  Note that even though the optimal
  rate for $\|u - u_h\|_{1,h,\eps}$ is only $\mathcal{O}(h^{k})$,
  \eqref{eq:9} gives a {\em superconvergent} rate of
  $\mathcal{O}(h^{k+1})$ for $\| u_h - I_{V_h}u \|_{1,h,\eps}$.
\end{remark}

\begin{theorem}[Pressure robustness] \label{the::presrobust}
  Under the same assumptions as Theorem~\ref{the::errorconvergence}, 
  \begin{align*}
    \frac{1}{\nu}\Sigmanormh{ \Stressvar - \Stressvarh}
    &+ \| \omega - \omega_h \| + \| u_h - I_{V_h}u \|_{V_h}
      \lesssim h^s \| (0, \omega, \sigma, 0)\|_{\nu, s}.
  \end{align*}
\end{theorem}
\begin{proof}
  Proceeding along the lines of the proof of
  Theorem~\ref{the::errorconvergence}, omitting the pressure error,
  we obtain, instead of~\eqref{eq:7}, 
  \begin{align*}
    \frac{1}{\nu}
     \| \Stressvar - \Stressvarh \|
    + \| (e_h^u, e_h^\omega) \|_{U_h}
      \lesssim 
      \frac{h^s}{\nu}
      \| \sigma\|_{H^s(\mesh)} 
      +
      \frac{1}{\sqrt{\nu}}
      \|(e_h^u,  e_h^\omega, e_h^\sigma, 0) \|_*.
  \end{align*}
  We may now complete the proof as before by
  using~\eqref{eq:big-inf-sup-2} 
  instead of~\eqref{eq:big-inf-sup-1}.
\end{proof}
%In \cite{gopledschoe}

\section{Postprocessing} \label{sec::postprocessing}

In this section we describe and analyze a postprocessing for the
discrete velocity.  While for the raw solution $u_h$, we may only
expect $\|u - u_h\|_{1,h,\eps}$ to go to zero at the rate
$\mathcal{O}(h^{k}),$ we will show that a locally postprocessed
velocity $u_h^*$ has error $\|u - u_h^*\|_{1,h,\eps}$ that converges
to zero at the higher rate $\mathcal{O}(h^{k+1})$ for sufficiently
regular solutions.  The key to obtain this enhanced accuracy, as
in~\cite{Stenberg1988}, is the $O(h^{k+1})$-superconvergence of
$\| u_h - I_{V_h}u\|_{1,h,\eps}$ -- see
Remark~\ref{rem:superconvergence}.  Finally, we shall also show that
$u_h^*$ retains the prized structure preservation properties of exact
mass conservation and pressure robustness.

The crucial ingredient is a reconstruction operator
(see~\cite{ledlehrschoe2017relaxedpartI,ledlehrschoe2018relaxedpartII})
whose properties are summarized in the next lemma.
Let
\begin{align*}
  V_h^*
  & =
    H_0(\divergence, \om) \cap
    \Poly^{k+1}(\mesh, \rr^d), \text{ and }
  \\
  V_h^{*,-}
  &=  \{ v_h \in
    \Poly^{k+1}(\mesh, \rr^d):\; \Pi^{k}\jump{(v_h)_n}=0, \text{ for all } F \in \facets \}
\end{align*}
denote the BDM space (one order higher) and its ``relaxed'' analogue,
respectively. The next result is a consequence of \cite[Lemmas~3.3
and~4.8]{ledlehrschoe2017relaxedpartI} and the Korn
inequality~\eqref{eq:discrete_Korn}.

\begin{lemma} \label{lem::reconstruction}
  There exists an operator $\mathcal{R}: V_h^{*,-} \rightarrow V_h^*,$
  whose application is computable element-by-element, satisfying
\begin{enumerate}
\item \label{item:lem::reconstruction-normbd}
  $\| \mathcal{R} v_h\|_{1,h,\eps} \lesssim \| v_h \|_{1,h,\eps},$ for
  al $v_h \in V_h^{*,-}$,
\item \label{item:lem::reconstruction-projection}
  $\mathcal{R} v_h^* = v_h^*$ for all $v_h^* \in V_h^*$,  and 
\item \label{item:lem::reconstruction-div0}
  whenever the local (element-wise) property $\divergence(v_h|_T) = 0$
  holds for all $T \in \mesh$ and all $v_h \in V_h^{*,-}$, the
  global property $\divergence(\mathcal{R}v_h) = 0$ holds.
\end{enumerate}
\end{lemma}

A simple choice of $\mathcal{R}$ is given by the classical BDM
intepolant. This was used in~\cite{guzman2016h}.  Another choice of
$\mathcal{R}$, given in~\cite{ledlehrschoe2017relaxedpartI}, based
on a simple averaging of coefficients, is significantly less
expensive for high orders.

The postprocessed solution $u_h^*\in V_h^*$ is given in two steps as
follows. First, using the computed $\sigma_h$ and $u_h$, solve the
{\em local} (see Remark~\ref{rem:postprocessing})
minimization problem
\begin{align} \label{eq::minimization}
  %u_T^h := \argmin\limits_{\substack{v_T^* \in V_h^*(T)\\ \Pi^k(v_T^* - u_h) \cdot n = 0 \textrm{ on } \partial T \\ \divergence(v_h^T) = 0} } \| \nu \eps(v_T^*) - \sigma_h \|^2_T
  u_h^{*,-} := \argmin\limits_{\substack{v_h^{*,-} \in V_h^{*,-}
  \\
  I_{V_h}(v_h^{*,-}) = u_h} } \| \nu \eps(v_h^{*,-}) - \sigma_h \|^2_T.
\end{align}
Second, apply the reconstruction and set $u_h^{*} :=
\mathcal{R}(u_h^{*,-})$.

% Note that there exists a solution of this minimization problem as for example $u_h$ fulfills the constraints (thus the admissible set is not empty) and the functional is convex.

% 
\begin{theorem}\label{the::postconvergence}
  Suppose the assumptions of Theorem~\ref{the::errorconvergence} hold.
  Then
  $u_h^* \in V_h^*,$ $\divergence(u_h^*) = 0,$ and for
  $s = \min(m-1,k+1)$ we have the pressure-robust error estimate
  \begin{align*}
    \| u - u_h^* \|_{1,h,\eps} \lesssim h^s
    \| (u, \omega, \sigma, 0)\|_{\nu, s}.
  \end{align*} 
\end{theorem}
\begin{proof}
  On any $T \in \mesh$, the condition $I_{V_h}(u_h^{*,-}) = u_h$
  implies that the Raviart-Thomas d.o.f.s applied to $u_h^{*,-}$ and
  $u_h$ coincide. Hence, for all $q_h \in \Poly^{k}(T, \rr)$,
  % \begin{align*}
  %   \int_T \divergence(u_h^{*,-}) q_h \dx &= -\int_T u_h^{*,-} \cdot \nabla q_h \dx+ \int_{\partial T} u_h^{*,-} \cdot n q_h \ds \\
  %   &= -\int_T u_h \cdot \nabla q_h \dx + \int_{\partial T} u_h \cdot n q_h \ds = \int_T \divergence(u_h) q_h \dx = 0
        %     \end{align*}
  \begin{align*}
    (\divergence(u_h^{*,-}), q_h)_T
    &=
      -(u_h^{*,-}, \nabla q_h)_T
      +(u_h^{*,-} \cdot n,  q_h)_{\partial T}
      \\
    &= -(u_h , \nabla q_h)_T
      + (u_h \cdot n,  q_h)_{\d T}
      = (\divergence(u_h), q_h) = 0 
  \end{align*}
  as $\divergence(u_h) = 0$. Thus, Lemma \ref{lem::reconstruction}
  implies that $u_h \in V^*_h$ and $\div(u_h^*)=0$.

  It only remains to prove the error estimate.  Let $I_{V_h^*}$ be the
  standard $\BDM^{k+1}$ interpolator. Then,
  $u_h^* = \mathcal{R} u_h^{*,-}$ satisfies 
  \begin{align*}
    \| u - u_h^* \|_{1,h,\eps}
    &\le
      \| u - I_{V_h^*}u \|_{1,h,\eps} + \|\mathcal{R} ( I_{V_h^*}u  -
      u_h^{*,-}) \|_{1,h,\eps}
    && \text{ by
       Lemma~\ref{lem::reconstruction}~(\ref{item:lem::reconstruction-projection}),}
    \\
    &\lesssim \| u - I_{V_h^*}u \|_{1,h,\eps} + \|u  -  u_h^{*,-}
      \|_{1,h,\eps}
          && \text{ by
       Lemma~\ref{lem::reconstruction}~(\ref{item:lem::reconstruction-normbd})}.
  \end{align*}
  Since standard approximation estimates yield
  $\| u - I_{V_h^*}u \|_{1,h,\eps}  \lesssim h^s
    \| (u, 0, 0, 0)\|_{\nu, s}$,  we focus on the last term.
  A triangle inequality (where we add and subtract different functions in the element and facet terms) yields
  \begin{equation}
    \label{eq::stepone}
    \begin{aligned} 
      \|u  -  u_h^{*,-} \|^2_{1,h,\eps} &\lesssim  \sum\limits_{T \in \mesh}\frac{1}{\nu^2} \| \nu\eps(u) - \sigma_h  \|^2_{T} + \sum\limits_{T \in \mesh}\frac{1}{\nu^2} \| \sigma_h - \nu\eps(u_h^{*,-}) \|^2_{T}\\
      &+\sum\limits_{F \in \facets} \frac{1}{h} \| \jump{ (u -I_{V_h^*}u)_t} \|_F^2 +\sum\limits_{F \in \facets} \frac{1}{h}\| \jump{ (I_{V_h^*}u -u_h^{*,-} )_t} \|_F^2. 
    \end{aligned}    
  \end{equation}
  Naming the four sums on the right as $s_1,s_2, s_3$ and $s_4$,
  respectively, we proceed to estimate each.
  Obviously $s_1 = \nu^{-1} \| \sigma - \sigma_h\|
  \lesssim h^s
    \| (0, \omega, \sigma, 0)\|_{\nu, s}$ by 
  Theorem~\ref{the::presrobust}.

  To bound $s_2$, note that for any $w_h$ in the admissible set of the
  minimization problem~\eqref{eq::minimization}, we have
  $s_2 \le \nu^{-2}\| \sigma_h - \nu \eps(w_h) \|^2$. We choose
  $w_h = I_{V_h^*}u + u_h - I_{V_h}u \in V_h^* \subset V_h^{*,-}$.
  Since $I_{V_h} I_{V_h^*}u = I_{V_h}u$ implies  $I_{V_h} w_h = u_h$,
  the chosen $w_h$ is in the admissible set. Hence,
  \begin{align*}
   s_2
    & \le  \nu^{-2} \| \sigma_h - \nu\eps(w_h) \|^2
     \le
      \nu^{-2} \big( \| \sigma_h - \nu\eps(I_{V_h^*}u) \|
      +
      \| \nu\eps(u_h) - \nu\eps(I_{V_h}u) \| \big)^2
    \\
    &\lesssim
      \nu^{-2}  \| \sigma_h - \nu\eps(u) \|^2 +
      \nu^{-2}\| \nu\eps(u) - \nu\eps(I_{V_h^*}u) \|^2 +
      \nu^{-2}\| \nu\eps(u_h) - \nu\eps(I_{V_h}u) \|^2
    \\
    & = \nu^{-2} \| \sigma_h - \sigma\|^2 +
      \| u - I_{V_h^*}u \|_{1,h,\eps}^2 + \| u_h - I_{V_h}u \|_{1,h,\eps}^2,
  \end{align*}
  so a standard approximation estimate and
  Theorem~\ref{the::presrobust} yield
  $s_2 \lesssim h^s  \| (u, \omega, \sigma, 0)\|_{\nu, s}.$

  The same standard approximation estimate for $I_{V_h^*}$ also gives
  $s_3 \le \| u - I_{V_h^*} u \|_{1,h,\eps} \lesssim h^s \| (u,
  \omega, \sigma, 0)\|_{\nu, s}$. Hence it only remains to bound
  $s_4$. Observe that
  $I_{V_h^*}u -u_h^{*,-} = I_{V_h} (I_{V_h^*}u -u_h^{*,-} ) + (\id -
  I_{V_h}) (I_{V_h^*}u -u_h^{*,-} ) = (I_{V_h}u - u_h) + (\id -
  I_{V_h}) (I_{V_h^*}u -u_h^{*,-} ),$ because
  $I_{V_h} I_{V_h^*}u = I_{V_h}u$ and $I_{V_h} u_h^{*,-} = u_h$.  This
  implies, letting
  $a = (\id - I_{V_h}) (\id - \Pi^{{\RM}})(I_{V_h^*}u -u_h^{*,-} )$,
  the identity $I_{V_h^*}u -u_h^{*,-} = (I_{V_h}u - u_h) + a$ holds
  because $(\id - I_{V_h}) \RM =0$ (as $k \ge 1$). Hence
  \begin{align}
    \label{eq:11}
    s_4
    & \lesssim
      \| I_{V_h}u  - u_h \|_{1,h,\eps}^2
      +
      \sum\limits_{F \in \facets}
      h^{-1} \big\| \jump{ a_t }\big\|_F^2.
  \end{align}
  Since the first term can be bounded by
  Theorem~\ref{the::presrobust}, let us consider the last term.  On
  any facet $F$ adjacent to a mesh element $T$, a trace inequality
  yields
  $h^{-1} \big\| \jump{ a_t} \big\|_F^2 \le h^{-1} \| a_t \|_{\partial
    T}^2 \lesssim \| \nabla a \|_{T}^2 + h^{-2} \| a \|_{T}^2.$ Hence, 
  \begin{align*}
    h^{-1} \big\| \jump{ a_t} \big\|_F^2
    & \lesssim
      \| \nabla (\id - \Pi^{{\RM}})(I_{V_h^*}u -u_h^{*,-} ) \|_{T}^2 + h^{-2} \| (\id - \Pi^{{\RM}})(I_{V_h^*}u -u_h^{*,-} ) \|_{T}^2 \\
                                                       &\lesssim \| \eps (I_{V_h^*}u -u_h^{*,-} ) \|_{T}^2 
  \end{align*}
  where we have used the continuity properties of $I_{V_h}$, scaling
  arguments, \eqref{eq::kornone}, and an estimate analogous
  to~\eqref{eq::korntwo}. Using triangle inequality and 
  returning to~\eqref{eq:11}, 
  \begin{align*}
    s_4
    & \lesssim
      \| I_{V_h}u  - u_h \|_{1,h,\eps}^2
      + \| \eps (I_{V_h^*}u - u) \|_h^2 + \nu^{-2}\| \nu \eps (u) -
      \sigma_h\|_h^2 + \nu^{-2}\| \sigma_h - \nu\eps (u_h^{*,-} ) \|_h^2.
  \end{align*}
  The last two terms are $s_1$ and $s_2$, respectively. Hence the
  prior estimates, the standard approximation estimate for
  $I_{V_h^*}$, and Theorem~\ref{the::presrobust} shows
  $s_4 \lesssim h^s \| (u, \omega, \sigma, 0)\|_{\nu, s}.$
\end{proof}

\begin{remark}
  \label{rem:postprocessing}
  The restriction of the minimizer of~\eqref{eq::minimization} to an
  element $T$, namely $u_T^{*,-} := u_h^{*,-}|_T,$ can be computed
  using the following Euler-Lagrange equations.  Letting
  $ \Lambda_h^*(T) = \{ \lambda: \lambda|_F \in \Poly^k(F, \rr)$ on
  all facets $F \subset \partial T\}$, the function $u_T^{*,-}$ is the
  unique function in $\Poly^{k+1}(T, \rr^d)$, which together with
  $\ell_h^* \in \Poly^{k-1}(T, \rr^d)$ and $\lambda_h^* \in \Lambda_h^*(T)$,
  satisfies
  \begin{alignat*}{2}
    ( \nu\eps(u_T^*), \eps(v) )_T +
    (\ell_h^*,  v)_T + 
    (\lambda_h^*, v \cdot n)_{\d T}
    &= (\sigma_h,   \eps(v) )_T,
    \\
    (u_T^*, \wp)_T
    &= (u_h, \wp)_T, \\
    (u^*_T \cdot n, \mu)_{\d T} &=   ( u_h \cdot n, \mu)_{\d T},
  \end{alignat*}
  for all $v \in \Poly^{k+1}(T, \rr^d),$
  $\wp \in \Poly^{k-1}(T, \rr^d)$ and $\mu \in \Lambda_h^*(T)$. The
  last two equations are another way to express the constraint
  $I_{V_h} u_h^{*,-} = u_h$ in~\eqref{eq::minimization}.
\end{remark}

\section{Numerical exampels}

In this last section we present two numerical examples to verify our
method. All examples were implemented within the finite element
library NGSolve/Netgen, see \cite{netgen,schoeberl2014cpp11} and on
\url{www.ngsolve.org}. The computational domain is given by $\Omega =
[0,1]^d$ and  the velocity field is driven by the volume force determined by $f = -\div(\Stressvar) + \nabla \Presvar$ with the exact solution given by
\begin{align*}
  \Stressvar &= \visc\eps(\curl(\psi_2)),\quad \textrm{and} \quad  \Presvar := x^5 + y^5 - \frac{1}{3} \quad \textrm{for } d=2\\
  \Stressvar &= \visc\eps(\curl(\psi_3,\psi_3,\psi_3)), \quad \textrm{and} \quad  \Presvar := x^5 + y^5 +z^5 - \frac{1}{2} \quad \textrm{for } d=3.
\end{align*}
Here $\psi_2 :=x^2(x-1)^2y^2(y-1)^2$ and $\psi_3 :=x^2(x-1)^2y^2(y-1)^2z^2(z-1)^2$ defines a given potential in two and three dimensions respectively and we choose the viscosity $\nu = 10^{-3}$.

In Tables \ref{twodexample} and \ref{threedexample} we report the
errors in all the computed solution components for varying polynomial
orders $k=1,2,3$ in the two and the three dimensional cases,
respectively. As predicted by Theorem \ref{the::errorconvergence} and
Theorem \ref{the::postconvergence} the corresponding errors converge
at optimal order. Furthermore, the $L^2$-norm of error of the (postprocessed) velocity error converges at one order higher. % This can be derived by the standard Aubin-Nitsche duality argument whenever the problem admit full elliptic regularity and the exact solution is smooth enough. 
Note that in three dimensions the errors are already quite small
already on the coarsest mesh. It appears that to get out of the 
preasymptotic regime and see the proper convergence rate, it
takes several steps.

\begin{table}
\begin{center}
\begin{subtable}{0.9\textwidth}
\footnotesize
%\begin{tabular}{r@{~}|@{~}c@{(}c@{)~}c@{(}c@{)}c@{(}c@{)~}c@{(}c@{)}}
\begin{tabular}{@{~}c@{~}|@{~}c@{~(}c@{)~}c@{~(}c@{)~}c@{~(}c@{)~}c@{~(}c@{)~}c@{~(}c@{)}}
      \toprule
  $|\mathcal{T}|$
  & $\| \nabla u - \nabla u_h^*\|_h$ & \footnotesize eoc
  &$\| u - u_h^*\|$ & \footnotesize eoc
  &$\| \sigma - \sigma_h\|$ & \footnotesize eoc
  &$\| \Presvar - \Presvarh\|$ & \footnotesize eoc
  &$\| \omega - \omega_h\|$ & \footnotesize eoc \\
\midrule
    \multicolumn{11}{c}{$k=1$}\\
20& \num{0.009902863275354638}&--& \num{0.0008392229939034794}&--& \num{0.010312466960396744}&--& \num{0.03441303228214327}&--& \num{0.008827275037994896}&--\\
80& \num{0.0035305629762117435}&\numeoc{1.4879474697552844}& \num{0.0001650894804110187}&\numeoc{2.345806013025177}& \num{0.0035791831556474005}&\numeoc{1.5266872201599668}& \num{0.009362753882968158}&\numeoc{1.8779501800952547}& \num{0.003224753279108695}&\numeoc{1.4527793631604784}\\
320& \num{0.0009503922416875517}&\numeoc{1.8933032867992963}& \num{2.3944303017503037e-05}&\numeoc{2.7854938472650708}& \num{0.0009408141443608543}&\numeoc{1.927648716739806}& \num{0.002377404964960747}&\numeoc{1.9775452630604289}& \num{0.0009240590715195726}&\numeoc{1.8031318007779626}\\
1280& \num{0.00025220378957418294}&\numeoc{1.9139331055354694}& \num{3.4001782223723697e-06}&\numeoc{2.8160001683316476}& \num{0.0002456179479367673}&\numeoc{1.9374937640549346}& \num{0.000596785141652696}&\numeoc{1.9941041496296363}& \num{0.00025736678061946184}&\numeoc{1.8441592294592353}\\
5120& \num{6.533973632118895e-05}&\numeoc{1.9485574156383147}& \num{4.6099406232374353e-07}&\numeoc{2.882790294319957}& \num{6.295522096878068e-05}&\numeoc{1.9640180531731288}& \num{0.00014935080989370941}&\numeoc{1.9985065557483128}& \num{6.863349183158698e-05}&\numeoc{1.9068411901094666}\\
    \midrule
  \multicolumn{11}{c}{$k=2$}\\
  20& \num{0.0022233180532418803}&--& \num{0.00010038685221944823}&--& \num{0.0018075725335161032}&--& \num{0.0037231340211480194}&--& \num{0.00145840604060015}&--\\
80& \num{0.0005032146154294715}&\numeoc{2.1434686158245717}& \num{1.0585431526209894e-05}&\numeoc{3.245418341613661}& \num{0.0003722622209445382}&\numeoc{2.2796624234705822}& \num{0.0005311576281535504}&\numeoc{2.8093055808259533}& \num{0.0002768736975362225}&\numeoc{2.397092529929767}\\
320& \num{6.655823868997209e-05}&\numeoc{2.9184846631391377}& \num{7.744534306321447e-07}&\numeoc{3.7727577836868944}& \num{5.0634161619243626e-05}&\numeoc{2.878136242351472}& \num{6.747790680282963e-05}&\numeoc{2.9766529373451784}& \num{4.1425908068076595e-05}&\numeoc{2.740622779780947}\\
1280& \num{8.355809373311833e-06}&\numeoc{2.9937657716700166}& \num{4.9323804751557606e-08}&\numeoc{3.972822493503286}& \num{6.371723935662283e-06}&\numeoc{2.9903553977156916}& \num{8.471067211643275e-06}&\numeoc{2.993799579853649}& \num{5.196616887090514e-06}&\numeoc{2.9948887146160157}\\
5120& \num{1.0432950155797786e-06}&\numeoc{3.0016324089703033}& \num{3.0810025950848047e-09}&\numeoc{4.0008122883395485}& \num{7.958407260327208e-07}&\numeoc{3.001132127398458}& \num{1.0600376182881933e-06}&\numeoc{2.998428272644474}& \num{6.445731725877459e-07}&\numeoc{3.0111566531914167}\\
    \midrule
  \multicolumn{11}{c}{$k=3$}\\
  20& \num{0.00041461110840970884}&--& \num{1.4432041688597988e-05}&--& \num{0.00023764732982608342}&--& \num{7.196988373849738e-05}&--& \num{0.00022403283910605226}&--\\
80& \num{4.783867967564853e-05}&\numeoc{3.1155092926538948}& \num{8.436872051505858e-07}&\numeoc{4.096423378271388}& \num{2.6976394394346334e-05}&\numeoc{3.13905275296886}& \num{5.700689004084194e-06}&\numeoc{3.658185123986187}& \num{2.627106790817884e-05}&\numeoc{3.0921634686810022}\\
320& \num{2.9560264215510807e-06}&\numeoc{4.01644650265318}& \num{2.6062418171114704e-08}&\numeoc{5.016665368596337}& \num{1.7295967880532224e-06}&\numeoc{3.96318987957517}& \num{3.647565093356667e-07}&\numeoc{3.966130669308764}& \num{1.711084861445255e-06}&\numeoc{3.940491629291239}\\
1280& \num{1.8617196676820212e-07}&\numeoc{3.9889514070850076}& \num{8.299954785406924e-10}&\numeoc{4.972723661372874}& \num{1.1152577173196006e-07}&\numeoc{3.9549867126645517}& \num{2.2927675254735363e-08}&\numeoc{3.991771738493527}& \num{1.133819783070296e-07}&\numeoc{3.91564806024593}\\
5120& \num{1.1724979142493852e-08}&\numeoc{3.98897859098612}& \num{2.6319348601454175e-11}&\numeoc{4.978907789079624}& \num{7.09771100495445e-09}&\numeoc{3.973879487149761}& \num{1.4350127545221598e-09}&\numeoc{3.9979546157977186}& \num{7.320933788024666e-09}&\numeoc{3.953019860065994}\\
\bottomrule
\end{tabular}
\caption{The $d=2$  example.} \label{twodexample}
\end{subtable}
\begin{subtable}{0.9\textwidth}  
\footnotesize
\begin{tabular}{@{~}c@{~}|@{~}c@{~(}c@{)~}c@{~(}c@{)~}c@{~(}c@{)~}c@{~(}c@{)~}c@{~(}c@{)}}
      \toprule
  $|\mathcal{T}|$
  & $\| \nabla u - \nabla u_h^*\|_h$ & \footnotesize eoc
  &$\| u - u_h^*\|$ & \footnotesize eoc
  &$\| \sigma - \sigma_h\|$ & \footnotesize eoc
  &$\| \Presvar - \Presvarh\|$ & \footnotesize eoc
  &$\| \omega - \omega_h\|$ & \footnotesize eoc \\
\midrule
  \multicolumn{11}{c}{$k=1$}\\
  28& \num{0.0015349129378525051}&--& \num{0.00013553036365765282}&--& \num{0.0014620066907744457}&--& \num{0.07450493074963317}&--& \num{0.0010571518970404754}&--\\
224& \num{0.0008111782170569502}&\numeoc{0.9200660098628043}& \num{5.416915520176647e-05}&\numeoc{1.3230726067108716}& \num{0.0008146666083537369}&\numeoc{0.8436682325324645}& \num{0.03105944573135321}&\numeoc{1.2623038220843439}& \num{0.0006697554943967613}&\numeoc{0.6584762689821013}\\
1792& \num{0.0003168837271946305}&\numeoc{1.3560653360341148}& \num{1.3159266227038611e-05}&\numeoc{2.0413925460548366}& \num{0.0003161871108793768}&\numeoc{1.3654312174979688}& \num{0.009518865089765845}&\numeoc{1.706170604451079}& \num{0.0003168724896876829}&\numeoc{1.0797320984784329}\\
14336& \num{9.195769681796024e-05}&\numeoc{1.7849113377403427}& \num{1.9253870348450005e-06}&\numeoc{2.772858659503397}& \num{8.982613775685127e-05}&\numeoc{1.8155713496221784}& \num{0.002533417862807191}&\numeoc{1.90970451961227}& \num{9.052034989304395e-05}&\numeoc{1.8075883466576945}\\
  114688& \num{2.384372161076711e-05}&\numeoc{1.9473608983973878}& \num{2.4849826742903404e-07}&\numeoc{2.9538407824513633}& \num{2.3142731339484638e-05}&\numeoc{1.9565761603207215}& \num{0.0006437098670992293}&\numeoc{1.9766025668256861}& \num{2.3436236065100434e-05}&\numeoc{1.9495012742403233}\\
    \midrule
  \multicolumn{11}{c}{$k=2$}\\
  28& \num{0.0005013890165629336}&--& \num{4.301511717938578e-05}&--& \num{0.0005763220104374992}&--& \num{0.006749483578828532}&--& \num{0.0004883840518419042}&--\\
224& \num{0.00020758834235332232}&\numeoc{1.2722049650201008}& \num{9.650726925770082e-06}&\numeoc{2.1561342475704035}& \num{0.0001579508339086626}&\numeoc{1.8673995651730226}& \num{0.0015501927278646272}&\numeoc{2.1223295320478717}& \num{0.00013523658294044473}&\numeoc{1.8525306211201085}\\
1792& \num{5.6951789868017195e-05}&\numeoc{1.8659123393405297}& \num{1.511987759247539e-06}&\numeoc{2.674191155366064}& \num{3.86875476364732e-05}&\numeoc{2.029534368029918}& \num{0.00026190251526978425}&\numeoc{2.5653457695744977}& \num{3.568187228628573e-05}&\numeoc{1.9222222459072054}\\
14336& \num{7.872544751941488e-06}&\numeoc{2.8548392247594516}& \num{1.058019998719267e-07}&\numeoc{3.8370076571554796}& \num{5.4193991689378905e-06}&\numeoc{2.8356644620603197}& \num{3.530495399346918e-05}&\numeoc{2.8910873732903175}& \num{5.243883677105361e-06}&\numeoc{2.766483729494196}\\
114688& \num{1.037903593272394e-06}&\numeoc{2.9231576097409904}& \num{6.995959166944021e-09}&\numeoc{3.9187012185811074}& \num{7.146729465521086e-07}&\numeoc{2.9227778313628456}& \num{4.496196263433237e-06}&\numeoc{2.973093719646426}& \num{7.018157184226596e-07}&\numeoc{2.901471518250042}\\
      \midrule
  \multicolumn{11}{c}{$k=3$}\\
  28& \num{0.00017591628643765745}&--& \num{1.2756623768693696e-05}&--& \num{0.00016684040492181786}&--& \num{0.0023674942267368614}&--& \num{0.00012673879302556537}&--\\
224& \num{5.75300018887787e-05}&\numeoc{1.6125026318842968}& \num{2.418206577408277e-06}&\numeoc{2.3992371503510994}& \num{4.4264989142323186e-05}&\numeoc{1.9142307460384556}& \num{0.00025140542569457335}&\numeoc{3.235273216944435}& \num{2.9834428178739327e-05}&\numeoc{2.086808150588493}\\
1792& \num{6.806454821350378e-06}&\numeoc{3.079339054092517}& \num{1.6824021060772411e-07}&\numeoc{3.8453430270335116}& \num{4.953349067192706e-06}&\numeoc{3.1596898698639224}& \num{2.9807358037288535e-05}&\numeoc{3.0762753721138627}& \num{3.622045410075866e-06}&\numeoc{3.0421015869801256}\\
14336& \num{5.74368577219584e-07}&\numeoc{3.566854829473553}& \num{7.31258795418667e-09}&\numeoc{4.523996678859669}& \num{4.111428734556125e-07}&\numeoc{3.590692565376722}& \num{2.0511984589008857e-06}&\numeoc{3.8611295200623945}& \num{3.0152815208714015e-07}&\numeoc{3.586440023341312}\\
114688& \num{3.9810139442644746e-08}&\numeoc{3.850768993929522}& \num{2.4647606368871644e-10}&\numeoc{4.89086261967543}& \num{2.7562896489240112e-08}&\numeoc{3.898840413207374}& \num{1.310788376662412e-07}&\numeoc{3.9679603929985845}& \num{2.034661277284434e-08}&\numeoc{3.889432160032821}\\
\bottomrule
\end{tabular}
\caption{The $d=3$ example.} \label{threedexample}
\end{subtable}
\caption{Convergence rates for the postprocessed velocity and all other
  solution components for $\nu = 10^{-3}$} \label{tab:2d3d}
\end{center}
\end{table}

\bibliographystyle{siam}
\bibliography{literature}
\end{document}

%% file: macros.tex
\newcommand{\argmin}{\mathop{\mathrm{argmin}}}
\newcommand{\divergence}{\operatorname{div}}

\newcommand{\eps}{\varepsilon}

\newcommand{\id}{\operatorname{Id}}

\newcommand{\curl}{\operatorname{curl}}

\newcommand{\rr}{\mathbb{R}}

\newcommand{\vecb}[1]{{#1}}
\newcommand{\matb}[1]{{#1}}

\newcommand{\trans}{{\textrm{T}}}

\newcommand{\ecoord}{{\vecb{x}}}

\newcommand{\normal}{{\vecb{n}}}
\newcommand{\tangential}{{\vecb{t}}}

\newcommand{\ds}{\mathop{~\mathrm{d} s}}

\newcommand{\dx}{\mathop{~\mathrm{d} \ecoord}}

\newcommand{\mesh}{\mathcal{T}_h}

\newcommand{\facets}{\mathcal{F}_h}
\newcommand{\facetsint}{\mathcal{F}_h^{\text{int}}}
\newcommand{\facetsext}{\mathcal{F}_h^{\text{ext}}}

\newcommand{\trace}[1]{\textrm{tr}({#1})}
\newcommand{\jump}[1]{ {[\![ #1 ] \!]}}
\newcommand{\Dev}[1]{\operatorname{dev}{\!(#1)}}

\newcommand{\visc}{\nu}

\newcommand{\blfone}{b_{1}}
\newcommand{\ablf}{a}
\newcommand{\blftwo}{b_{2}}
\newcommand{\blfbig}{B}
\newcommand{\Velspace}{{\vecb{V}}}
\newcommand{\Velspaceh}{{\Velspace_h}}

\newcommand{\Velvar}{{\vecb{u}}}
\newcommand{\Velvartest}{{\vecb{v}}}
\newcommand{\Velvarh}{{\vecb{u}_h}}
\newcommand{\Velvarhtest}{\vecb{v}_h}

\newcommand{\Velnormh}[1]{|| #1 ||_{U_h}}

\newcommand{\Sigmaplus}{{\Sigma_h^{+}}}
\newcommand{\Stressspace}{{\matb{\Sigma}}}
\newcommand{\Stressspaceh}{{\Stressspace_h}}
\newcommand{\Stressvar}{{\matb{\sigma}}}
\newcommand{\Stressvartest}{{\matb{\tau}}}
\newcommand{\Stressvarh}{\matb{\sigma}_h}

\newcommand{\Stressvarhtest}{\matb{\tau}_h}

\newcommand{\Sigmanormh}[1]{\left\| #1 \right\|}

\newcommand{\Presspace}{{Q}}
\newcommand{\Presspaceh}{{\Presspace_h}}
\newcommand{\Presvar}{{p}}
\newcommand{\Presvartest}{{q}}
\newcommand{\Presvarh}{{p_h}}
\newcommand{\Presvarhtest}{{q_h}}

\newcommand{\Presnormh}[1]{|| #1 ||}

\newcommand{\Hcurldiv}[1]{H(\curl\divergence,#1)}

\newcommand{\Poly}{{\mathbb{P}}}

\newcommand{\Forcevar}{{\vecb{f}}}

\newcommand{\RRR}{\mathbb{R}}
\newcommand{\om}{\Omega}

\newcommand{\kernel}{K_h}

\newcommand{\proj}{\Pi}

\newcommand{\piola}{\mathcal{P}}

\newcommand{\covariantpiola}{ \mathcal{M}}
\newcommand{\MM}{\covariantpiola}

\newcommand{\dd}{\mathbb{D}}

\newcommand{\DD}{\mathcal{D}}
\let\d\partial
\newcommand{\grad}{\nabla}
\renewcommand{\div}{\operatorname{div}}
\newcommand{\dt}{{\tilde{d}}}
\newcommand{\ip}[1]{\langle {#1} \rangle}

\newcommand{\RT}{{\mathcal{RT}}}
\newcommand{\RM}{{\mathbb{E}}}
\newcommand{\BDM}{{\mathcal{BDM}}}
\newcommand{\skw}{\mathop\text{skw}}
\newcommand{\kk}{\mathbb{K}}
\newcommand{\mm}{\mathbb{M}}
\newcommand{\vv}{\mathbb{V}}

\newcommand{\hatcurl}{\hat{\curl}}
\newcommand{\ckorn}{c_K}

\def\addlegendimage{\csname pgfplots@addlegendimage\endcsname}

\pgfkeys{/pgfplots/number in legend/.style={%
        /pgfplots/legend image code/.code={%
            \node at (0.295,-0.0225){#1};
        },%
    },
}

%% file: main.bbl
\begin{thebibliography}{10}

\bibitem{MR840802}
{\sc D.~N. Arnold, F.~Brezzi, and J.~Douglas, Jr.}, {\em P{EERS}: a new mixed
  finite element for plane elasticity}, Japan J. Appl. Math., 1 (1984),
  pp.~347--367.

\bibitem{MR2336264}
{\sc D.~N. Arnold, R.~S. Falk, and R.~Winther}, {\em Mixed finite element
  methods for linear elasticity with weakly imposed symmetry}, Math. Comp., 76
  (2007), pp.~1699--1723.

\bibitem{MR2449101}
{\sc D.~Boffi, F.~Brezzi, and M.~Fortin}, {\em Reduced symmetry elements in
  linear elasticity}, Commun. Pure Appl. Anal., 8 (2009), pp.~95--121.

\bibitem{brezzi2012mixed}
\leavevmode\vrule height 2pt depth -1.6pt width 23pt, {\em Mixed Finite Element
  Methods and Applications}, Springer Science \& Business Media, 2013.

\bibitem{MR2047078}
{\sc S.~C. Brenner}, {\em Korn's inequalities for piecewise {$H^1$} vector
  fields}, Math. Comp., 73 (2004), pp.~1067--1087.

\bibitem{brezzi1985two}
{\sc F.~Brezzi, J.~Douglas~Jr., and L.~D. Marini}, {\em Two families of mixed
  finite elements for second order elliptic problems}, Numerische Mathematik,
  47 (1985), pp.~217--235.

\bibitem{CockbGopal04}
{\sc B.~Cockburn and J.~Gopalakrishnan}, {\em A characterization of hybridized
  mixed methods for the {D}irichlet problem}, SIAM J. Numer. Anal., 42 (2004),
  pp.~283--301.

\bibitem{MR2629995}
{\sc B.~Cockburn, J.~Gopalakrishnan, and J.~Guzm\'{a}n}, {\em A new elasticity
  element made for enforcing weak stress symmetry}, Math. Comp., 79 (2010),
  pp.~1331--1349.

\bibitem{cockburn2005locally}
{\sc B.~Cockburn, G.~Kanschat, and D.~Sch{\"o}tzau}, {\em A locally
  conservative {LDG} method for the incompressible {Navier}-{Stokes}
  equations}, Mathematics of Computation, 74 (2005), pp.~1067--1095.

\bibitem{cockburn2007note}
\leavevmode\vrule height 2pt depth -1.6pt width 23pt, {\em A note on
  discontinuous {Galerkin} divergence-free solutions of the {Navier}--{Stokes}
  equations}, Journal of Scientific Computing, 31 (2007), pp.~61--73.

\bibitem{MR3194122}
{\sc B.~Cockburn and F.-J. Sayas}, {\em Divergence-conforming {HDG} methods for
  {S}tokes flows}, Math. Comp., 83 (2014), pp.~1571--1598.

\bibitem{Farhloulcanadian}
{\sc M.~Farhloul}, {\em Mixed and nonconforming finite element methods for the
  stokes problem}, Canadian Applied Mathematics Quarterly, 3 (Fall 1995).

\bibitem{MR1231323}
{\sc M.~Farhloul and M.~Fortin}, {\em A new mixed finite element for the
  {S}tokes and elasticity problems}, SIAM J. Numer. Anal., 30 (1993),
  pp.~971--990.

\bibitem{MR1464150}
{\sc M.~Farhloul and M.~Fortin}, {\em Dual hybrid methods for the elasticity
  and the {S}tokes problems: a unified approach}, Numer. Math., 76 (1997),
  pp.~419--440.

\bibitem{MR1934446}
{\sc M.~Farhloul and M.~Fortin}, {\em Review and complements on mixed-hybrid
  finite element methods for fluid flows}, in Proceedings of the 9th
  {I}nternational {C}ongress on {C}omputational and {A}pplied {M}athematics
  ({L}euven, 2000), vol.~140, 2002, pp.~301--313.

\bibitem{fujinqiu}
{\sc G.~Fu, Y.~Jin, and W.~Qiu}, {\em Parameter-free superconvergent
  h(div)-conforming hdg methods for the brinkman equations}, IMA Journal of
  Numerical Analysis,  (2018), p.~dry001.

\bibitem{GopalGuzma12}
{\sc J.~Gopalakrishnan and J.~Guzm{\'a}n}, {\em A second elasticity element
  using the matrix bubble}, {IMA} J. Numer. Anal., 32 (2012), pp.~352--372.

\bibitem{gopledschoe}
{\sc J.~Gopalakrishnan, P.~L. Lederer, and J.~Sch\"{o}berl}, {\em {A mass
  conserving mixed stress formulation for the Stokes equations}}, Preprint
  arXiv:1806.07173,  (2018).

\bibitem{guzman2016h}
{\sc J.~Guzm{\'a}n, C.-W. Shu, and F.~A. Sequeira}, {\em H (div) conforming and
  dg methods for incompressible euler’s equations}, IMA Journal of Numerical
  Analysis,  (2016), p.~drw054.

\bibitem{stenberg2011}
{\sc J.~K\"onn\"o and R.~Stenberg}, {\em Numerical computations with
  {H(div)}-finite elements for the {Brinkman} problem}, Computational
  Geosciences, 16 (2012), pp.~139--158.

\bibitem{ledlehrschoe2017relaxedpartI}
{\sc P.~L. Lederer, C.~Lehrenfeld, and J.~Sch\"{o}berl}, {\em {Hybrid
  Discontinuous {Galerkin} methods with relaxed {$H(div)$}-conformity for
  incompressible flows. Part I}}, to appear in SIAM journal on numerical
  analysis (preprint arXiv:1707.02782),  (2017).

\bibitem{ledlehrschoe2018relaxedpartII}
\leavevmode\vrule height 2pt depth -1.6pt width 23pt, {\em {Hybrid
  Discontinuous {Galerkin} methods with relaxed {$H(div)$}-conformity for
  incompressible flows. Part II}}, to appear in ESAIM: M2AN (preprint
  arXiv:1805.06787),  (2018).

\bibitem{2016arXiv160903701L}
{\sc P.~L. Lederer, A.~Linke, C.~Merdon, and J.~Sch\"oberl}, {\em
  Divergence-free {R}econstruction {O}perators for {P}ressure-{R}obust {S}tokes
  {D}iscretizations with {C}ontinuous {P}ressure {F}inite {E}lements}, SIAM J.
  Numer. Anal., 55 (2017), pp.~1291--1314.

\bibitem{LedererSchoeberl2017}
{\sc P.~L. Lederer and J.~Sch\"oberl}, {\em {Polynomial robust stability
  analysis for $H$(div)-conforming finite elements for the Stokes equations}},
  IMA Journal of Numerical Analysis,  (2017), p.~drx051.

\bibitem{LS_CMAME_2016}
{\sc C.~Lehrenfeld and J.~Sch\"{o}berl}, {\em High order exactly
  divergence-free hybrid discontinuous galerkin methods for unsteady
  incompressible flows}, Computer Methods in Applied Mechanics and Engineering,
  307 (2016), pp.~339 -- 361.

\bibitem{RaviaThoma77}
{\sc P.-A. Raviart and J.~M. Thomas}, {\em A mixed finite element method for
  2nd order elliptic problems}, in Mathematical aspects of finite element
  methods (Proc. Conf., Consiglio Naz. delle Ricerche (C.N.R.), Rome, 1975),
  Springer, Berlin, 1977, pp.~292--315. Lecture Notes in Math., Vol. 606.

\bibitem{MR0483555}
\leavevmode\vrule height 2pt depth -1.6pt width 23pt, {\em A mixed finite
  element method for 2nd order elliptic problems},  (1977), pp.~292--315.
  Lecture Notes in Math., Vol. 606.

\bibitem{netgen}
{\sc J.~Sch{\"o}berl}, {\em {NETGEN An advancing front 2D/3D-mesh generator
  based on abstract rules}}, Computing and Visualization in Science, 1 (1997),
  pp.~41--52.

\bibitem{schoeberl2014cpp11}
{\sc J.~Sch\"oberl}, {\em C++11 implementation of finite elements in
  {NGSolve}}, Tech. Rep. ASC-2014-30, Institute for Analysis and Scientific
  Computing, September 2014.

\bibitem{Stenberg1988}
{\sc R.~Stenberg}, {\em A family of mixed finite elements for the elasticity
  problem}, Numerische Mathematik, 53 (1988), pp.~513--538.

\end{thebibliography}
